\documentclass[a4paper]{amsart}
\usepackage[english]{babel}
\usepackage{amssymb}
\usepackage[initials]{amsrefs}
\usepackage{tikz}

\usepackage[arrow,matrix,curve]{xy}

\allowdisplaybreaks

\newcommand{\into}{{\triangleright}}

\newcommand{\op}{\operatorname}
\newcommand{\co}{\colon}
\newcommand{\id}{\mathrm{id}}
\newcommand{\NN}{\mathbb{N}}

\newcommand{\ZZ}{\mathbb{Z}}
\newcommand{\CC}{\mathbb{C}}

\newcommand{\calB}{\mathcal{B}}
\newcommand{\calC}{\mathcal{C}}

\newcommand{\calM}{\mathcal{M}}
\newcommand{\calN}{\mathcal{N}}
\newcommand{\calO}{\mathcal{O}}
\newcommand{\calP}{\mathcal{P}}
\newcommand{\calQ}{\mathcal{Q}}
\newcommand{\calR}{\mathcal{R}}
\newcommand{\calS}{\mathcal{S}}

\newtheorem{thm}{Theorem}[section]
\newtheorem{lem}[thm]{Lemma}
\newtheorem{cor}[thm]{Corollary}
\newtheorem{prop}[thm]{Proposition}
\theoremstyle{definition}
\newtheorem{defi}[thm]{Definition}

\theoremstyle{remark}
\newtheorem{rem}[thm]{Remark}

\title{$L^2$--invisibility of symmetric operad groups}
\author{Werner Thumann}
\address{Karlsruhe Institute of Technology, Karlsruhe, Germany}

\begin{document}

\begin{abstract}
	We show a homological result for the class of planar or symmetric operad groups: We show that under certain 
	conditions, group (co)homology of such groups with certain coefficients vanishes in all dimensions, provided it vanishes 
	in dimension $0$. This can be applied for example to $l^2$--homology or cohomology with coefficients in the group ring. 
	As a corollary, we obtain explicit vanishing results for Thompson-like groups such as the Brin--Thompson groups $nV$.
\end{abstract}
\maketitle
%\tableofcontents

\section{Introduction}

In \cite{sa-th:laa} it is shown that a certain class of groups acting on compact ultrametric spaces, the so-called
dually contracting local similarity groups, are $l^2$--invisible. The latter means that group homology with group von Neumann
algebra coefficients vanishes in every dimension, i.e.
\[H_k\big(G,\calN(G)\big)=0\]
for all $k\geq 0$ where $\calN(G)$ denotes the group von Neumann algebra of $G$. If $G$ is of type $F_\infty$, i.e.~there is
a classifying space for $G$ with finitely many cells in each dimension, then this is equivalent to
\[H_k\big(G,l^2(G)\big)=0\]
for all $k\geq 0$ by \cite{lue:lta}*{Lemmas 6.98 on p.~286 and 12.3 on p.~438}.

In \cite{thu:ogatfp} the author proposed to study fundamental groups of categories naturally associated to operads. This class
of groups, called operad groups, contains a lot of Thompson-like groups already existent in the literature. Among these are the
above mentioned local similarity groups (see \cite{thu:ogatfp}*{Subsection 3.5}).

This article is mainly concerned with generalizing the results of \cite{sa-th:laa} to the setting of symmetric operad groups 
which form a much larger class of groups. The proof in \cite{sa-th:laa} consisted of constructing a suitable simplicial complex
on which the group in question acts and then applying a spectral sequence associated to this action which computes the
homology of the group in terms of the homology of the stabilizer subgroups. The proof in the case of operad groups goes exactly
the same way. However, it is a priori unclear how to construct the simplicial complex. The reason is the following: A local 
similarity group is defined as a representation, i.e.~as a group of homeomorphisms of a compact ultrametric space. This space is 
used to construct the simplicial complex as a poset of partitions of this space. The case of operad groups is more abstract. A priori,
there is no canonical space comparable to these ultrametric spaces on which an operad group acts. However, these spaces, called 
limit spaces, are conjectured to exist if the operad satisfies the calculus of fractions 
(see \cite{thu:ogatfp}*{Subsection 3.3} for the latter notion). We don't use these limit spaces here. Instead, we will take 
the conjectured correspondence between calculus of fractions operads and their limit spaces as a motivation to mimic the necessary
notions for the construction of the desired simplicial complex in terms of the operad itself.

As in \cite{sa-th:laa}, we briefly want to discuss the relationship between these results and Gromov's
\emph{Zero-in-the-spectrum conjecture} (see \cite{gro:lrm}). The algebraic version of this conjecture states that if $\Gamma=\pi_1(M)$
is the fundamental group of a closed aspherical Riemannian manifold, then there always exists a dimension $p\geq 0$ such
that $H_p(\Gamma,\calN\Gamma)\neq 0$ or equivalently $H_p(\Gamma,l^2\Gamma)\neq 0$. Conjecturally, the fundamental groups of 
closed aspherical manifolds are precisely the Poincar\'e duality groups $G$ of type $F$, i.e.~there is a compact classifying space
for $G$ and a natural number $n\geq 0$ such that
\[H^i(G,\ZZ G)=\begin{cases}0&\text{if }i\neq n\\\ZZ&\text{if }i=n\end{cases}\]
(see \cite{dav:pdg}). Dropping Poincar\'e duality and relaxing type $F$ to type $F_\infty$, we arrive at a more general question which
has been posed by L\"uck in \cite{lue:lta}*{Remark 12.4 on p.~440}: If $G$ is a group of type $F_\infty$, does there always exist a $p$
with $H_p(G,\calN G)\neq 0$? In \cite{thu:ogatfp} we discuss conditions for operads which imply that the associated operad groups 
are of type $F_\infty$. Combining this with the results in the present article, we obtain a large class of groups of type $F_\infty$ 
which are also $l^2$--invisible. This class contains the well-known symmetric Thompson group $V$ and consequently, L\"uck's question 
has to be answered in the negative. Unfortunately, all these groups $G$ are neither of type $F$ nor satisfy Poincar\'e duality since, 
as another corollary of our main theorem (Theorem \ref{20869}), we can show $H^k(G,\ZZ G)=0$ for all $k\geq 0$.

\subsection{Prerequisites} 

The present article is based on Sections 2 and 3 of \cite{thu:ogatfp}.

\subsection{Notation and Conventions}

When $f\co A\rightarrow B$ and $g\co B\rightarrow C$ are two composable arrows, we write $f*g$ for the composition $A\rightarrow C$
instead of the usual notation $g\circ f$. Consequently, it is often better to plug in arguments from the left. When we do this, 
we use the notation $x\into f$ for the evaluation of $f$ at $x$. However, we won't entirely drop the usual notation $f(x)$ and use 
both notations side by side. Objects of type $\operatorname{Aut}(X)$ will be made into a group by the definition $fg=f\cdot g:=f*g$.
Conversely, a group $G$ is considered as a groupoid with one object and arrows the elements in $G$ together with the composition
$f*g:=f\cdot g$.

\subsection{Acknowledgments}

I want to thank my PhD adviser Roman Sauer for the opportunity to pursue mathematics, for his guidance, encouragement and
support over the last few years.
I also gratefully acknowledge financial support by the DFG grants 1661/3-1 and 1661/3-2.

\section{Statement of the main theorem}

\begin{defi}
	Let $\calM G$ be a $\ZZ G$--module for every group $G$. We call $\calM$
	\begin{itemize}
		\item {\it K\"unneth} if for every two groups $G_1,G_2$ and $n_1,n_2\in\ZZ$ with $n_i\geq -1$ the following is satisfied:
		\[\left.\begin{array}{r}\forall_{k\leq n_1}\ H_k(G_1,\calM G_1)=0\\\forall_{k\leq n_2}\ H_k(G_2,\calM G_2)=0\end{array}
			\right\}\ \Longrightarrow\ \forall_{k\leq n}\ H_k(G,\calM G)=0\]
		where $G:=G_1\times G_2$ and $n:=n_1+n_2+1$.
		\item {\it inductive} if whenever $H$ and $G$ are groups with $H$ a subgroup of $G$ and $k\geq 0$, then we have that
		\[H_k(H,\calM H)=0\hspace{3mm}\text{implies}\hspace{3mm}H_k(H,\calM G)=0\]
	\end{itemize}
	Let $\mathfrak{P}$ be a property of groups. Then we say that $\calM$ is $\mathfrak{P}$--K\"unneth if the 
	property K\"unneth has to be satisfied only for $\mathfrak{P}$--groups $G_1,G_2$. We say that $\calM$ is $\mathfrak{P}$--inductive
	if the property inductive has to be satisfied only for $\mathfrak{P}$--subgroups $H$ of the arbitrary group $G$.
	Furthermore, one can formulate these two properties also in the cohomological case.
\end{defi}

\begin{defi}
	Let $\calO$ be a planar or symmetric or braided operad and $X$ an object in $\calS(\calO)$. We say that $X$ is
	\begin{itemize}
		\item {\it split} if there are objects $A_1,A_2,A_3$ and an arrow $A_1\otimes X\otimes A_2\otimes X\otimes A_3\rightarrow X$
		in $\calS(\calO)$.
		\item {\it progressive} if for every arrow $Y\rightarrow X$ there are objects $A_1,A_2$ and an arrow
		$A_1\otimes X\otimes A_2\rightarrow Y$ such that the coordinates of $X$ are connected to only one
		operation in this arrow (see Figure \ref{57690}).
	\end{itemize}
\end{defi}

\begin{figure}[!ht]
\begin{center}\begin{tikzpicture}[scale=0.65]
	\draw (0,4.6) -- +(-1,0);
	\draw (0,5.4) -- +(-1,0);
	\draw (0,4.1) -- +(0,1.8) -- +(1.8,0.9) -- +(0,0);
	\draw (0,3) -- +(-1,0);
	\draw (0,2.1) -- +(0,1.8) -- +(1.8,0.9) -- +(0,0);
	\draw (0,0.4) -- +(-1,0);
	\draw (0,1) -- +(-1,0);
	\draw (0,1.6) -- +(-1,0);
	\draw (0,0.1) -- +(0,1.8) -- +(1.8,0.9) -- +(0,0);
	\draw[dotted] (-0.8,0) -- +(0,6);
	\draw[white, line width=4pt] (-1,1.6) .. controls (-2,1.6) and (-3,3.5) .. (-4,3.5);
	\draw (-1,1.6) .. controls (-2,1.6) and (-3,3.5) .. (-4,3.5);
	\draw[white, line width=4pt] (-1,1) .. controls (-2,1) and (-3,0.5) .. (-4,0.5);
	\draw (-1,1) .. controls (-2,1) and (-3,0.5) .. (-4,0.5);
	\draw[white, line width=4pt] (-1,4.6) .. controls (-2,4.6) and (-3,1.5) .. (-4,1.5);
	\draw (-1,4.6) .. controls (-2,4.6) and (-3,1.5) .. (-4,1.5);
	\draw[white, line width=4pt] (-1,3) .. controls (-2,3) and (-3,5.5) .. (-4,5.5);
	\draw (-1,3) .. controls (-2,3) and (-3,5.5) .. (-4,5.5);
	\draw[white, line width=4pt] (-1,0.4) .. controls (-2,0.4) and (-3,2.5) .. (-4,2.5);
	\draw (-1,0.4) .. controls (-2,0.4) and (-3,2.5) .. (-4,2.5);
	\draw[white, line width=4pt] (-1,5.4) .. controls (-2,5.4) and (-3,4.5) .. (-4,4.5);
	\draw (-1,5.4) .. controls (-2,5.4) and (-3,4.5) .. (-4,4.5);
	\draw  (-4.8,5.6) rectangle (-4.2,4.4);
	\draw  (-4.8,3.6) rectangle (-4.2,2.4);
	\draw  (-4.8,1.6) rectangle (-4.2,0.4);
	\draw  (2,5.2) rectangle (2.6,0.8);
	\node at (-4.5,5) {$A_1$};
	\node at (-4.5,3) {$X$};
	\node at (-4.5,1) {$A_2$};
	\node at (2.3,3) {$Y$};
\end{tikzpicture}\end{center}
\caption{An arrow $A_1\otimes X\otimes A_2\rightarrow Y$ such that $X$ is only connected to one operation.}\label{57690}
\end{figure}
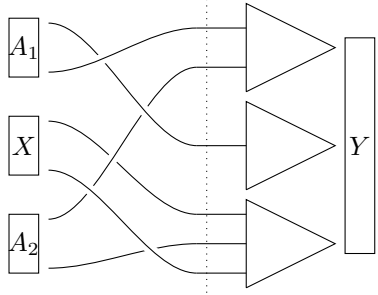

\begin{rem}\label{92437}
	If $X$ is just a single color, then $X$ is split if and only if there is an operation with output color $X$ and and at 
	least two inputs of color $X$. If $\calO$ is monochromatic and $X\neq I$ is an object of $\calS(\calO)$, then $X$ is 
	split if and only if there is at least one operation in $\calO$ with at least two inputs. So in the monochromatic case, 
	the condition split is in fact a property of $\calO$.
\end{rem}

\begin{rem}\label{71138}
	If $X$ is just a single color, then $X$ is progressive if and only if for every operation $\theta$ with output
	color $X$ there is another operation $\phi$ with at least one input of color $X$ and at least one input of 
	$\theta$ has the same color as the output of $\phi$. Now assume that $\calO$ is monochromatic. Then an object
	$X\neq I$ in $\calS(\calO)$ (which is just a natural number $X>0$, e.g.~$X=3$) is progressive if and only if there 
	is an operation in $\calO$ with at least $X$ inputs (e.g.~$3$ inputs). Note that $X=1$ is always progressive
	in the monochromatic case.
\end{rem}

\begin{thm}\label{20869}
	Let $\calO$ be a planar or symmetric operad which satisfies the calculus of fractions.
	Let $\calM$ be a coefficient system which is K\"unneth and inductive. Let $X$ be a split progressive object of
	$\calS(\calO)$. Set $\Gamma:=\pi_1(\calO,X)$. Then
	\[H_0(\Gamma,\calM\Gamma)=0\hspace{5mm}\Longrightarrow\hspace{5mm}\forall_{k\geq 0}\ H_k(\Gamma,\calM\Gamma)=0\]
	The same is true for cohomology.\\
	More generally, let $\mathfrak{P}$ be a property of groups which is closed under taking products. 
	Then the statement is true also for coefficient systems $\calM$ which are only $\mathfrak{P}$--K\"unneth and 
	$\mathfrak{P}$--inductive, provided that $\Gamma$ satisfies $\mathfrak{P}$.
\end{thm}

\begin{rem}\label{50091}
	Let $X,Y$ be objects in $\calS(\calO)$. Generalizing the notion of progressiveness, we say that $X$ is 
	$Y$--progressive if for every arrow $Z\rightarrow X$ there is an arrow $A_1\otimes Y\otimes A_2\rightarrow Z$
	and the coordinates of $Y$ are connected to only one operation in this arrow (call this the \emph{link condition}). 
	In particular, there is an arrow $A_1\otimes Y\otimes A_2\rightarrow X$.
	
	With this notion, we can formulate a slightly more general version of Theorem \ref{20869}: Let $\calO,\mathfrak{P},\calM$ be 
	as in the theorem. Let $X$ be an object of $\calS(\calO)$ and set $\Gamma=\pi_1(\calO,X)$. Assume there is a
	split object $Y$ such that $X$ is $Y$--progressive, $\Upsilon:=\pi_1(\calO,Y)$ satisfies $\mathfrak{P}$ and 
	$H_0(\Upsilon,\calM\Upsilon)=0$. Then $H_k(\Gamma,\calM\Gamma)=0$ for each $k\geq 0$. The same is true for 
	cohomology.
\end{rem}

\section{Proof of the main theorem}\label{96238}

We start with two general lemmas concerning the calculus of fractions.

\begin{lem}\label{25010}
	Let $\calC$ be a category satisfying the calculus of fractions. Then two square fillings of a given span
	can be combined to a common square filling. This means: Let $x,y$ be two arrows with the 
	same codomain and assume having two square fillings as in the diagram
	
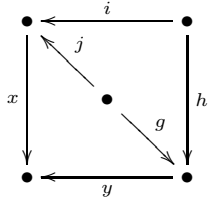
\begin{displaymath}\xymatrix@-5pt{
		\bullet\ar[dd]_x & & \bullet\ar[ll]_i\ar[dd]^h \\
		& \bullet\ar[ul]_j\ar[dr]^g & \\
		\bullet & & \bullet\ar[ll]^y
	}\end{displaymath}
	then we can complete this diagram to the commutative diagram
	
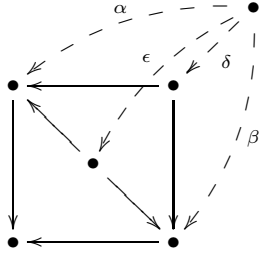
\begin{displaymath}\xymatrix@-5pt{
		& & & \bullet\ar@{-->}@/_10pt/[dlll]_\alpha\ar@{-->}[dl]^\delta
		\ar@{-->}@/_8pt/[ddll]_\epsilon\ar@{-->}@/^10pt/[dddl]^\beta \\
		\bullet\ar[dd] & & \bullet\ar[ll]\ar[dd] & \\
		& \bullet\ar[ul]\ar[dr] & & \\
		\bullet & & \bullet\ar[ll] &
	}\end{displaymath}
\end{lem}
\begin{proof}
	Let $c,d$ be a square filling of $a:=ix=hy,b:=jx=gy$, i.e.~$ca=db$. Then $ch$ and $dg$ are two parallel
	arrows which are coequalized by $y$, i.e.~$(ch)y=(dg)y$. By the equalization property we find an equalizing
	arrow $k$ with $k(ch)=k(dg)$. By the same reasoning we find an arrow $l$ with $l(ci)=l(dj)$. Let $m,n$ be a
	square filling of $l,k$, i.e.~$ml=nk=:p$. Then one can easily calculate that the arrows
	\[\delta=pc\hspace{8mm}\alpha=pci\hspace{8mm}\beta=pch\hspace{8mm}\epsilon=pd\]
	fill the diagram as required.
\end{proof}

\begin{lem}\label{69343}
	Let $\calC$ be a category satisfying the calculus of fractions. Let $\bar{x}$ and $\bar{y}$ be two arrows
	$A\rightarrow C$. Assume that there are arrows $x,y\co A\rightarrow B$ and $a\co B\rightarrow C$ such that
	$xa=\bar{x}$ and $ya=\bar{y}$.
	\begin{displaymath}\xymatrix{
		 C & B\ar[l]^a & A\ar[r]_y\ar@/^10pt/[rr]^{\bar{y}}\ar[l]^x\ar@/_10pt/[ll]_{\bar{x}} & B\ar[r]_a & C
	}\end{displaymath}
	Then the span $C\xleftarrow{\bar{x}}A\xrightarrow{\bar{y}}C$ is null-homotopic if and only if
	the span $B\xleftarrow{x}A\xrightarrow{y}B$ is null-homotopic.
\end{lem}
\begin{proof}
	First note that a span like $B\xleftarrow{x}A\xrightarrow{y}B$ is null-homotopic if and only if the
	parallel arrows $x$ and $y$ are homotopic. Since $\calC$ satisfies the calculus of fractions, this is the 
	case if and only if there is an equalizing arrow, i.e.~an arrow $d\co D\rightarrow A$ with $dx=dy$. Now if 
	$x$ and $y$ are homotopic then clearly also $\bar{x}$ and $\bar{y}$ are homotopic. On the other hand, assume that 
	$\bar{x}$ and $\bar{y}$ are homotopic and $d\co D\rightarrow A$ equalizes $\bar{x}$ and $\bar{y}$. Then we
	have
	\[(dx)a=d(xa)=d\bar{x}=d\bar{y}=d(ya)=(dy)a\]
	Then by the equalization property we find an arrow $e\co E\rightarrow D$ with $e(dx)=e(dy)$. Consequently, 
	the arrow $ed$ equalizes $x$ and $y$ and thus, $x$ and $y$ are homotopic.
\end{proof}

We now turn to the proof of Theorem \ref{20869}. In the following, let $\calO$ be a planar or symmetric 
operad satisfying the calculus of fractions with set of colors $C$ and let $\calS:=\calS(\calO)$.

\subsection{Marked objects} 

Let $c=(c_1,...,c_n)$ be an object of $\calS$, i.e.~$c_1,...,c_n$ are colors in $C$.
First we define a marking on $c$ in the symmetric case. It assigns to each coordinate of $c$ a 
symbol. A symbol can be assigned several times and not every coordinate has to be marked by a symbol. More precisely,
a marking of $c$ is a set $S$ of symbols together with a subset $I\subset\{1,...,n\}$ and a surjective function 
$f\co I\rightarrow S$. In the planar case, we additionally require the marking to be ordered. This means
that whenever $i\into f=j\into f$ for $i<j$ then also $i\into f=k\into f=j\into f$ for $i<k<j$.

Let $m_1,m_2$ be two markings of $c$ with symbol sets $S_1,S_2$. We say $m_1\subset m_2$ if there 
is a function $i\co S_1\rightarrow S_2$ and every coordinate which is marked with $s_1\in S_1$ is 
also marked with $s_1\into i\in S_2$. We say $m_1$ and $m_2$ are equivalent if $m_1\subset m_2$ and $m_2\subset m_1$.
This means that there is a bijection $i\co S_1\rightarrow S_2$ and a coordinate is marked with 
$s_1\in S_1$ if and only if it is marked with $s_1\into i\in S_2$. By slight abuse of notation, we identify equivalent
markings and write $m_1=m_2$ if they are equivalent. Then $\subset$ becomes a partial order on the set of markings
on $c$ (see the first paragraph of Subsection \ref{52176}).

\subsection{Marked arrows}

Let $\alpha\co c\rightarrow d$ be an arrow in $\calS$ with objects $c=(c_1,...,c_n)$ and
$d=(d_1,...,d_m)$. A marking on $\alpha$ is a marking on the domain $c$. A comarking on $\alpha$ is a marking
on the codomain $d$. A comarking on $\alpha$ induces a marking on $\alpha$: Let $(\sigma,X)$ be a representative of
$\alpha$ where $\sigma$ is either an identity or a colored permutation,
depending on whether $\calO$ is planar or symmetric. Write $X=(X_1,...,X_m)$.
The comarking yields a marking on the operations $X_i$: Mark each input of $X_i$
with the same symbol. Now push the markings through $\sigma$ to obtain a marking on the domain $c$.
Figure \ref{01354} illustrates this procedure. If $m$ is the comarking, then we denote this pull-backed
marking by $\alpha^*(m)$. Observe that this pull-back is functorial, i.e.~we have
\[(\alpha\beta)^*(m)=\alpha^*(\beta^*(m))\]
Furthermore, we have
\[m_1\subset m_2\ \ \Longleftrightarrow\ \ \alpha^*(m_1)\subset\alpha^*(m_2)\]

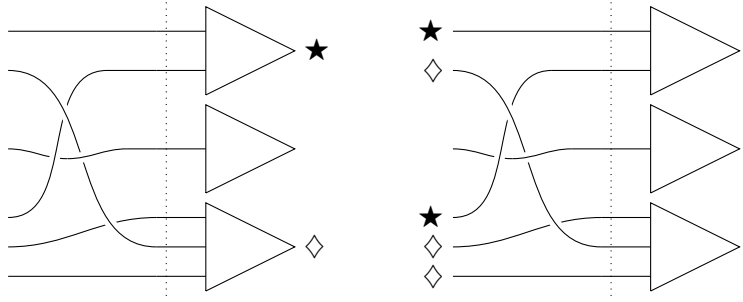
\begin{figure}[!ht]
\begin{center}\begin{tikzpicture}[scale=0.65]
	\draw (-1,5.4) to[out=left,in=right] (-4,5.4);
	\draw[white, line width=4pt] (-2.8,2.8) to[out=left,in=right] (-4,3);
	\draw (-2.8,2.8) to[out=left,in=right] (-4,3);
	\draw[white, line width=4pt] (-2,4.6) to[out=left,in=right] (-4,1.6);
	\draw (-2,4.6) to[out=left,in=right] (-4,1.6);
	\draw[white, line width=4pt] (-1,1.6) to[out=left,in=right] (-4,1);
	\draw (-1,1.6) to[out=left,in=right] (-4,1);
	\draw[white, line width=4pt] (-1,1) to[out=left,in=right] (-4,4.6);
	\draw (-1,1) to[out=left,in=right] (-4,4.6);
	\draw[white, line width=4pt] (-1.6,3) to[out=left,in=right] (-2.8,2.8);
	\draw (-1.6,3) to[out=left,in=right] (-2.8,2.8);
	\draw (-1,0.4) to[out=left,in=right] (-4,0.4);
	\draw (-2,4.6) -- (-1,4.6);
	\draw (-1.6,3) -- (-1,3);
	
	\draw (0,4.6) -- +(-1,0);
	\draw (0,5.4) -- +(-1,0);
	\draw (0,4.1) -- +(0,1.8) -- +(1.8,0.9) node[right]{$\bigstar$} -- +(0,0);
	\draw (0,3) -- +(-1,0);
	\draw (0,2.1) -- +(0,1.8) -- +(1.8,0.9) -- +(0,0);
	\draw (0,0.4) -- +(-1,0);
	\draw (0,1) -- +(-1,0);
	\draw (0,1.6) -- +(-1,0);
	\draw (0,0.1) -- +(0,1.8) -- +(1.8,0.9) node[right]{$\diamondsuit$} -- +(0,0);
	\draw[dotted] (-0.8,0) -- +(0,6);
	
	\draw (8,5.4) to[out=left,in=right] (5,5.4) node[left]{$\bigstar$};
	\draw[white, line width=4pt] (6.2,2.8) to[out=left,in=right] (5,3);
	\draw (6.2,2.8) to[out=left,in=right] (5,3);
	\draw[white, line width=4pt] (7,4.6) to[out=left,in=right] (5,1.6);
	\draw (7,4.6) to[out=left,in=right] (5,1.6) node[left]{$\bigstar$};
	\draw[white, line width=4pt] (8,1.6) to[out=left,in=right] (5,1);
	\draw (8,1.6) to[out=left,in=right] (5,1) node[left]{$\diamondsuit$};
	\draw[white, line width=4pt] (8,1) to[out=left,in=right] (5,4.6);
	\draw (8,1) to[out=left,in=right] (5,4.6) node[left]{$\diamondsuit$};
	\draw[white, line width=4pt] (7.4,3) to[out=left,in=right] (6.2,2.8);
	\draw (7.4,3) to[out=left,in=right] (6.2,2.8);
	\draw (8,0.4) to[out=left,in=right] (5,0.4) node[left]{$\diamondsuit$};
	\draw (7,4.6) -- (8,4.6);
	\draw (7.4,3) -- (8,3);
	
	\draw (9,4.6) -- +(-1,0);
	\draw (9,5.4) -- +(-1,0);
	\draw (9,4.1) -- +(0,1.8) -- +(1.8,0.9) -- +(0,0);
	\draw (9,3) -- +(-1,0);
	\draw (9,2.1) -- +(0,1.8) -- +(1.8,0.9) -- +(0,0);
	\draw (9,0.4) -- +(-1,0);
	\draw (9,1) -- +(-1,0);
	\draw (9,1.6) -- +(-1,0);
	\draw (9,0.1) -- +(0,1.8) -- +(1.8,0.9) -- +(0,0);
	\draw[dotted] (8.2,0) -- +(0,6);
\end{tikzpicture}\end{center}
\caption{A comarking (left) and the pull-backed marking (right).}\label{01354}
\end{figure}

\vspace{2mm}
\begin{center}
	\framebox[1.1\width]{Now fix an object $x$ in $\calS$.}
\end{center}
\vspace{2mm}

Let $(\alpha_1,m_1)$ and $(\alpha_2,m_2)$ be two marked arrows with codomain $x$, i.e.~$\alpha_i\co c_i\rightarrow x$ is 
an arrow and $m_i$ is a marking on $c_i$. We write
\[(\alpha_1,m_1)\subset(\alpha_2,m_2)\]
if there is a square filling
\begin{displaymath}\xymatrix{
	d\ar@{..>}[r]^{\beta_2}\ar@{..>}[d]_{\beta_1}&c_2\ar[d]^{\alpha_2}\\
	c_1\ar[r]_{\alpha_1}&x
}\end{displaymath}
with $\beta_1^*(m_1)\subset\beta_2^*(m_2)$. Observe that then {\it every} square filling satisfies this:
Let $(\gamma_1,\gamma_2)$ be another square filling of $(\alpha_1,\alpha_2)$. Then choose a
common square filling $(\delta_1,\delta_2)$ as in Lemma \ref{25010}. It is not hard to see that the property
$\delta_1^*(m_1)\subset\delta_2^*(m_2)$ is inherited from the square filling $(\beta_1,\beta_2)$. On the other
hand, this forces the property onto the square filling $(\gamma_1,\gamma_2)$, i.e.~we have $\gamma_1^*(m_1)\subset
\gamma_2^*(m_2)$.
\begin{rem}\label{00176}
	This observation implies also the following: Let $(\alpha_1,m_1)\subset(\alpha_2,m_2)$
	and assume that $\alpha_1=\alpha_2$. Then we necessarily have $m_1\subset m_2$. Indeed, we can
	choose $\beta_1=\id=\beta_2$ in the above square filling.
\end{rem}

\begin{prop}\label{53020}
	The relation $\subset$ on the set of marked arrows is reflexive and transitive.
\end{prop}
\begin{proof}
	Reflexivity is clear. For transitivity assume $(\alpha_1,m_1)\subset(\delta,m)$ and
	$(\delta,m)\subset(\alpha_2,m_2)$. Choose two square fillings
	\begin{displaymath}\xymatrix{
		a_1\ar@{..>}[d]_{\beta_1}\ar@{..>}[r]^{\delta_1}&d\ar[d]^{\delta}&
		a_2\ar@{..>}[l]_{\delta_2}\ar@{..>}[d]^{\beta_2}\\
		c_1\ar[r]_{\alpha_1}&x&c_2\ar[l]^{\alpha_2}
	}\end{displaymath}
	with $\beta_1^*(m_1)\subset\delta_1^*(m)$ and $\delta_2^*(m)\subset\beta_2^*(m_2)$. Choose a square filling of
	$(\delta_1,\delta_2)$
	\begin{displaymath}\xymatrix{
		&e\ar[dl]_{\gamma_1}\ar[d]^{\eta}\ar[dr]^{\gamma_2}&\\
		a_1\ar@{..>}[d]_{\beta_1}\ar@{..>}[r]_{\delta_1}&d\ar[d]^{\delta}&
		a_2\ar@{..>}[l]^{\delta_2}\ar@{..>}[d]^{\beta_2}\\
		c_1\ar[r]_{\alpha_1}&x&c_2\ar[l]^{\alpha_2}
	}\end{displaymath}
	Now we have
	\begin{eqnarray*}
		(\gamma_1\beta_1)^*(m_1)&=&\gamma_1^*(\beta_1^*(m_1))\\
								&\subset &\gamma_1^*(\delta_1^*(m))\\
								&=&(\gamma_1\delta_1)^*(m)\\
								&=&\eta^*(m)\\
								&=&(\gamma_2\delta_2)^*(m)\\
								&=&\gamma_2^*(\delta_2^*(m))\\
								&\subset &\gamma_2^*(\beta_2^*(m_2))\\
								&=&(\gamma_2\beta_2)^*(m_2)
	\end{eqnarray*}
	This proves $(\alpha_1,m_1)\subset(\alpha_2,m_2)$.
\end{proof}

\subsection{Balls and partitions}\label{52176}

A transitive and reflexive relation $\preccurlyeq$ on a set $Z$ is not a poset in general since $a\preccurlyeq b$ together with 
$b\preccurlyeq a$ does not imply $a=b$ in general. We can repair this in the following way: Define $a,b\in Z$ to be equivalent if 
$a\preccurlyeq b$ and $b\preccurlyeq a$. This is indeed an equivalence relation because $\preccurlyeq$ is assumed to be reflexive 
and transitive. Now if $\mathfrak{a}$ and $\mathfrak{b}$ are two equivalence classes, we write $\mathfrak{a}\leq\mathfrak{b}$
if there are representatives $a$ and $b$ respectively with $a\preccurlyeq b$. One can easily show that then \emph{any} two
representatives satisfy this. Using this, it is not hard to see that $\leq$ is indeed a partial order on the set of equivalence
classes. In particular, we have $\mathfrak{a}=\mathfrak{b}$ if and only if $\mathfrak{a}\leq\mathfrak{b}$ and 
$\mathfrak{b}\leq\mathfrak{a}$.

\vspace{2mm}

We want to apply this observation to the reflexive and transitive relation $\subset$ on the set of marked arrows.
We say that two marked arrows $(\alpha_1,m_1)$ and $(\alpha_2,m_2)$ with common codomain $x$ are equivalent if both 
$(\alpha_1,m_1)\subset(\alpha_2,m_2)$ and $(\alpha_2,m_2)\subset(\alpha_1,m_1)$ hold. We remark that this
is equivalent to the existence of a square filling
\begin{displaymath}\xymatrix{
	d\ar@{..>}[r]^{\beta_2}\ar@{..>}[d]_{\beta_1}&c_2\ar[d]^{\alpha_2}\\
	c_1\ar[r]_{\alpha_1}&x
}\end{displaymath}
with $\beta_1^*(m_1)=\beta_2^*(m_2)$ and moreover that {\it every} square filling satisfies this.
\begin{itemize}
	\item A {\it semi-partition} is an equivalence class of marked arrows.
	\item A {\it partition} is a semi-partition with fully marked domain for some (and therefore for every)
		representative of the semi-partition. Here, an object in $\calS$ is fully marked if every coordinate 
		is marked.
	\item A {\it multiball} is a semi-partition with an uni-marked domain for some (and therefore for every)
		representative of the semi-partition. Here, an object in $\calS$ is uni-marked if there is only one
		symbol in the marking.
	\item A {\it ball} is a semi-partition such that there is a single-marked representative. Here, an object in $\calS$ is 
		single-marked if only one coordinate is marked.
\end{itemize}
Note that these definitions depend on the base point $x$. Following the remarks in the first paragraph, we write
$\calP\subset\calQ$ for two semi-partitions $\calP$ and $\calQ$ if there are representatives $p$ of 
$\calP$ and $q$ of $\calQ$ satisfying $p\subset q$. Then for \emph{all} such representatives $p,q$
we have $p\subset q$. It follows that $\subset$ is a partial order on the set of semi-partitions. In particular,
we have $\calP=\calQ$ if and only if $\calP\subset\calQ$ and $\calQ\subset\calP$.

\vspace{2mm}

We now investigate the relationship between semi-partitions and multiballs. Let $\calP$ be a semi-partition with
representative $(\alpha,m)$. Picking out a symbol $s$ of $m$ and removing all markings except those with the chosen
symbol $s$ gives a uni-marked arrow $(\alpha,m^s)$. The corresponding equivalence class is a multiball and is
independent of the chosen representative $(\alpha,m)$ in the following sense: If we choose another representative
$(\beta,n)$, then $(\alpha,m)\sim(\beta,n)$ and to the chosen symbol $s$ of $m$ corresponds a unique symbol
$r$ of $n$. Deleting all markings in $n$ except those with the symbol $r$ gives a uni-marked arrow 
$(\beta,n^r)$ which is equivalent to $(\alpha,m^s)$. Multiballs arising in this way are called submultiballs of 
$\calP$ and we write $P\in\calP$ for submultiballs. Note that Remark \ref{00176} implies that two submultiballs 
$P_1,P_2$ coming from a representative of $\calP$ by choosing two different symbols satisfy $P_1\not\subset P_2$ and
$P_2\not\subset P_1$, in particular $P_1\neq P_2$. It follows that there is a canonical bijection between the set 
$\{P\in\calP\}$ of submultiballs of $\calP$ and the set of symbols of $\calP$ (which is, by definition, the set of
symbols of the marking of any representative for $\calP$). Moreover, any two submultiballs $P_1,P_2\in\calP$ with 
$P_1\neq P_2$ satisfy the stronger property $(P_1\not\subset P_2)\wedge(P_2\not\subset P_1)$. Equivalently, whenever
$P_1\subset P_2$ or $P_2\subset P_1$, then already $P_1=P_2$.

\begin{prop}\label{45613}
	Let $\calP,\calQ$ be semi-partitions, then 
	\[\calQ\subset\calP\ \ \Longleftrightarrow\ \ \forall_{Q\in\calQ}\ \exists_{P\in\calP}\ Q\subset P\]
	In particular, $\calP=\calQ$ if and only if $\{Q\in\calQ\}=\{P\in\calP\}$.
\end{prop}
\begin{proof}
	We first prove the last statement since it is a formal consequence of the previous statement and the remarks preceding 
	the proposition. First recall that $\calP=\calQ$ is equivalent to $\calP\subset\calQ$ and $\calQ\subset\calP$.
	The first statement of the proposition says
	that there is a function $i\co \{Q\in\calQ\}\rightarrow\{P\in\calP\}$ with the property that $Q\subset Q\into i$ 
	for each $Q\in\calQ$. Since we also have $\calP\subset\calQ$, there is another function
	$j\co \{P\in\calP\}\rightarrow\{Q\in\calQ\}$ with the property that $P\subset P\into j$ for each $P\in\calP$. We have
	\[Q\subset Q\into i\subset (Q\into i)\into j=Q\into(ij)\]
	for all $Q\in\calQ$. Since both the left and right side are submultiballs of $\calQ$, the remarks preceding
	the proposition imply $Q=Q\into(ij)$ for all $Q\in\calQ$. We then have
	\[Q\subset Q\into i\subset Q\]
	and therefore also $Q=Q\into i$ for all $Q\in\calQ$. This shows $\{Q\in\calQ\}\subset\{P\in\calP\}$. With a similar
	argument applied to $ji$, we also obtain $\{Q\in\calQ\}\supset\{P\in\calP\}$. So we have indeed 
	$\{Q\in\calQ\}=\{P\in\calP\}$. The converse implication also follows easily from the first 
	statement of this proposition.
	
	Now let's turn to the first statement: Assume $\calQ\subset\calP$. By the square filling technique, we know
	that we can choose a common arrow $\alpha\co c\rightarrow x$ with markings $m_\calQ\subset m_\calP$ such that
	$[\alpha,m_\calQ]=\calQ$ and $[\alpha,m_\calP]=\calP$. If $Q\in\calQ$, then we find a symbol $s_\calQ$ of the 
	marking $m_\calQ$ which corresponds to $Q$. But since $m_\calQ\subset m_\calP$, there is a unique symbol $s_\calP$
	of the marking $m_\calP$ such that the coordinates of $c$ marked by $s_\calQ$ are also marked by $s_\calP$.
	The submultiball obtained from $(\alpha,m_\calP)$ corresponding to the symbol $s_\calP$ is the one we are looking
	for.
	
	Conversely, assume that there is a function $i\co \{Q\in\calQ\}\rightarrow\{P\in\calP\}$ such that
	$Q\subset Q\into i$ for every $Q\in\calQ$. Using the square filling technique, we find a common arrow
	$\alpha\co c\rightarrow x$ with markings $m_\calQ,m_\calP$ such that $[\alpha,m_\calQ]=\calQ$ and 
	$[\alpha,m_\calP]=\calP$. We want to show $m_\calQ\subset m_\calP$. Let $s$ be any symbol of
	$m_\calQ$. To this symbol corresponds exactly one submultiball $Q\in\calQ$ such that $Q=[\alpha,m_\calQ^s]$
	where $m_\calQ^s$ is the submarking of $m_\calQ$ with all markings removed except those with the symbol $s$.
	To the submultiball $Q\into i\in\calP$ corresponds exactly one symbol $r$ of $m_\calP$ such that
	$Q\into i=[\alpha,m_\calP^r]$. Since $Q\subset Q\into i$ we have 
	$(\alpha,m_\calQ^s)\subset(\alpha,m_\calP^r)$ and therefore $m_\calQ^s\subset m_\calP^r$ by Remark
	\ref{00176}. It follows $m_\calQ\subset m_\calP$ and thus $\calQ\subset\calP$.
\end{proof}

\subsection{The action on the set of semi-partitions}

Here we will define an action of $\Gamma=\pi_1(\calS,x)$ on the set of semi-partitions over $x$. 
So let $\gamma\in\Gamma$
and $\calP$ be a semi-partition over $x$. We will define another semi-partition $\gamma\cdot\calP$ over $x$.
Recall that $\gamma$ is represented by a span $x\xleftarrow{\gamma_d}a\xrightarrow{\gamma_n}x$ (the $d$
refers to \emph{denominator} and the $n$ refers to \emph{nominator}) and that $\calP$ is represented by a marked arrow
$(\alpha\co c\rightarrow x,m)$. First choose a square filling of $(\gamma_n,\alpha)$
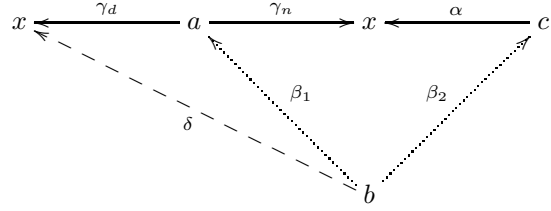
\begin{displaymath}\xymatrix{
	x&&a\ar[ll]_{\gamma_d}\ar[rr]^{\gamma_n}&&x&&c\ar[ll]_\alpha\\
	&&&&&&\\
	&&&&b\ar@{-->}[uullll]^{\delta}\ar@{..>}[uull]_{\beta_1}\ar@{..>}[uurr]^{\beta_2}&&
}\end{displaymath}
and then define $\delta:=\beta_1\gamma_d\co b\rightarrow x$. Endow this arrow with the marking $\mu:=\beta_2^*(m)$.
Finally, define $\gamma\cdot\calP:=[\delta,\mu]$. We have to show that this is well-defined, i.e.~we have to show 
that the resulting class is independent of
\begin{itemize}
	\item[\bf{1.}] the square filling $(\beta_1,\beta_2)$
	\item[\bf{2.}] the marked arrow $(\alpha,m)$ as a representative of $\calP$
	\item[\bf{3.}] the span $(\gamma_d,\gamma_n)$ as a representative of $\gamma$
\end{itemize}
We now prove these points one by one.

{\bf 1.} Assume we have two square fillings of $(\gamma_n,\alpha)$ as in the following diagram:
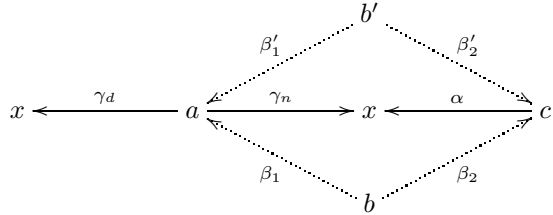
\begin{displaymath}\xymatrix{
	&&&&b'\ar@{..>}[dll]_{\beta'_1}\ar@{..>}[drr]^{\beta'_2}&&\\
	x&&a\ar[ll]_{\gamma_d}\ar[rr]^{\gamma_n}&&x&&c\ar[ll]_\alpha\\
	&&&&b\ar@{..>}[ull]^{\beta_1}\ar@{..>}[urr]_{\beta_2}&&
}\end{displaymath}
Choose a common square filling as in Lemma \ref{25010}:
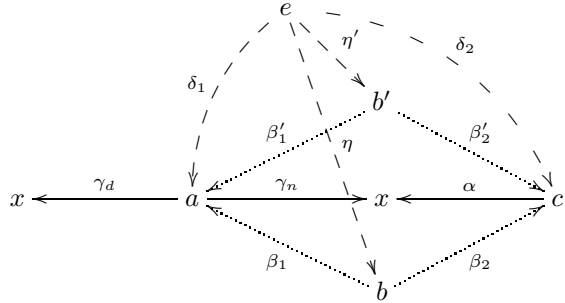
\begin{displaymath}\xymatrix{
	&&&e\ar@{-->}[dr]^{\eta'}\ar@{-->}[dddr]^{\eta}\ar@{-->}@/_10pt/[ddl]_{\delta_1}
	\ar@{-->}@/^20pt/[ddrrr]^{\delta_2}&&&\\
	&&&&b'\ar@{..>}[dll]_{\beta'_1}\ar@{..>}[drr]^{\beta'_2}&&\\
	x&&a\ar[ll]_{\gamma_d}\ar[rr]^{\gamma_n}&&x&&c\ar[ll]_\alpha\\
	&&&&b\ar@{..>}[ull]^{\beta_1}\ar@{..>}[urr]_{\beta_2}&&
}\end{displaymath}
Now  the marked arrow $(\beta_1\gamma_d,\beta_2^*(m))$ is equivalent to the marked arrow
$(\delta_1\gamma_d,\delta_2^*(m))$ via $\eta$. Analogously, the marked arrow $(\beta'_1\gamma_d,{\beta'_2}^*(m))$ is
equivalent to $(\delta_1\gamma_d,\delta_2^*(m))$ via $\eta'$ and therefore equivalent to 
$(\beta_1\gamma_d,\beta_2^*(m))$, q.e.d.

{\bf 2.} Let $(\alpha',m')$ be another marked arrow equivalent to $(\alpha,m)$ and choose a square filling
$(\beta,\beta')$ such that $\beta^*(m)={\beta'}^*(m')=:\mu$ as in the following diagram:
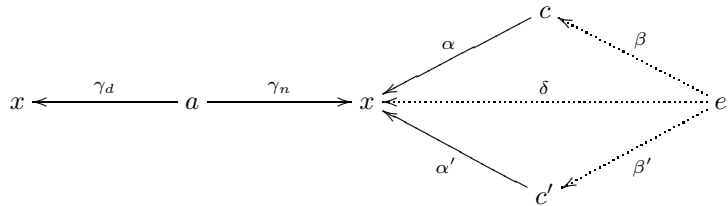
\begin{displaymath}\xymatrix{
	&&&&&&c\ar[dll]_{\alpha}&&\\
	x&&a\ar[ll]_{\gamma_d}\ar[rr]^{\gamma_n}&&x&&&&e\ar@{..>}[ull]_{\beta}\ar@{..>}[dll]^{\beta'}
	\ar@{..>}[llll]_{\delta}\\
	&&&&&&c'\ar[ull]^{\alpha'}&&
}\end{displaymath}
First choose a square filling $(\eta_1,\eta_2)$ of $(\gamma_n,\alpha)$ and then a square filling $(\nu_1,\nu_2)$
of $(\eta_2,\beta)$. Analogously, choose a square filling $(\eta'_1,\eta'_2)$ of $(\gamma_n,\alpha')$ and then a 
square filling $(\nu'_1,\nu'_2)$ of $(\eta'_2,\beta')$
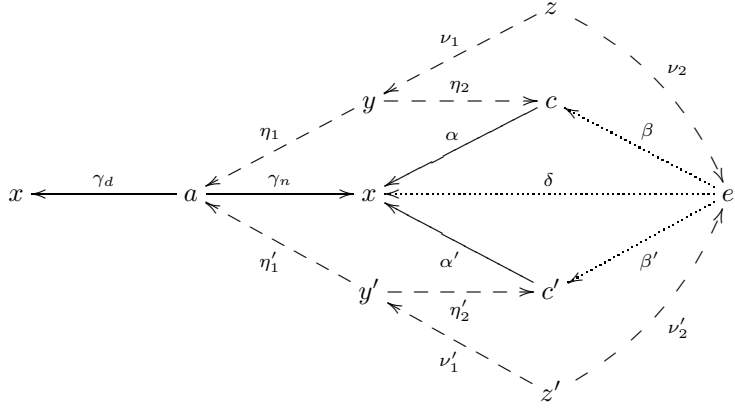
\begin{displaymath}\xymatrix{
	&&&&&&z\ar@{-->}@/^10pt/[ddrr]^{\nu_2}\ar@{-->}[dll]_{\nu_1}&&\\
	&&&&y\ar@{-->}[dll]_{\eta_1}\ar@{-->}[rr]^{\eta_2}&&c\ar[dll]_{\alpha}&&\\
	x&&a\ar[ll]_{\gamma_d}\ar[rr]^{\gamma_n}&&x&&&&e\ar@{..>}[ull]_{\beta}\ar@{..>}[dll]^{\beta'}
	\ar@{..>}[llll]_{\delta}\\
	&&&&y'\ar@{-->}[ull]^{\eta'_1}\ar@{-->}[rr]_{\eta'_2}&&c'\ar[ull]^{\alpha'}&&\\
	&&&&&&z'\ar@{-->}@/_10pt/[uurr]_{\nu'_2}\ar@{-->}[ull]^{\nu'_1}&&
}\end{displaymath}
The marked arrow $(\eta_1\gamma_d,\eta_2^*(m))$ is equivalent to $\Lambda:=(\nu_1\eta_1\gamma_d,\nu_2^*(\mu))$ 
via $\nu_1$. On the other side, the marked arrow $(\eta'_1\gamma_d,{\eta'_2}^*(m'))$ is equivalent to
$\Lambda':=(\nu'_1\eta'_1\gamma_d,{\nu'_2}^*(\mu))$ via $\nu'_1$. The marked arrows $\Lambda$ and $\Lambda'$ are both
constructed from the same marked arrow $(\delta,\mu)$ and so are equivalent by {\bf 1}. Consequently,
$(\eta_1\gamma_d,\eta_2^*(m))$ and $(\eta'_1\gamma_d,{\eta'_2}^*(m'))$ are equivalent, q.e.d.

{\bf 3.} Let $(\gamma'_d,\gamma'_n)$ be another representing span of $\gamma$ homotopic to the span
$(\gamma_d,\gamma_n)$. Then recall that the two spans can be filled by a diagram as follows:
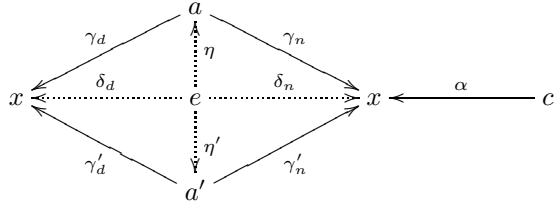
\begin{displaymath}\xymatrix{
	&&a\ar[dll]_{\gamma_d}\ar[drr]^{\gamma_n}&&&&\\
	x&&e\ar@{..>}[u]_{\eta}\ar@{..>}[d]^{\eta'}\ar@{..>}[rr]^{\delta_n}\ar@{..>}[ll]_{\delta_d}&&x&&c\ar[ll]_\alpha\\
	&&a'\ar[ull]^{\gamma'_d}\ar[urr]_{\gamma'_n}&&&&
}\end{displaymath}
Now choose a square filling $(\nu_1,\nu_2)$ of $(\delta_n,\alpha)$ and note that $(\epsilon,\nu_2)$, 
where $\epsilon:=\nu_1\eta$, gives a square filling of $(\gamma_n,\alpha)$.
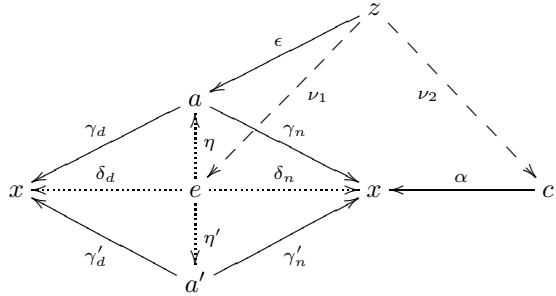
\begin{displaymath}\xymatrix{
	&&&&z\ar@{-->}[ddll]^/-8pt/{\nu_1}\ar@{-->}[ddrr]_/-8pt/{\nu_2}\ar[dll]_\epsilon&&\\
	&&a\ar[dll]_{\gamma_d}\ar[drr]^{\gamma_n}&&&&\\
	x&&e\ar@{..>}[u]_{\eta}\ar@{..>}[d]^{\eta'}\ar@{..>}[rr]^{\delta_n}\ar@{..>}[ll]_{\delta_d}&&x&&c\ar[ll]_\alpha\\
	&&a'\ar[ull]^{\gamma'_d}\ar[urr]_{\gamma'_n}&&&&
}\end{displaymath}
The marked arrow $(\epsilon\gamma_d,\nu_2^*(m))$ is equivalent to $(\nu_1\delta_d,\nu_2^*(m))$. Similarly, 
define $\epsilon'=\nu_1\eta'$ and note that $(\epsilon',\nu_2)$ gives a square filling of $(\gamma'_n,\alpha)$.
Again, the marked arrow $(\epsilon'\gamma'_d,\nu_2^*(m))$ is equivalent to $(\nu_1\delta_d,\nu_2^*(m))$. Therefore,
$(\epsilon\gamma_d,\nu_2^*(m))$ and $(\epsilon'\gamma'_d,\nu_2^*(m))$ are equivalent, q.e.d.

\vspace{2mm}

Now we want to show that this is indeed an action, i.e.~$1\cdot\calP=\calP$ and
$\gamma^1\cdot(\gamma^2\cdot\calP)=(\gamma^1\gamma^2)\cdot\calP$. The first property is easy to see. The
second property is not entirely trivial but straightforward. We will be explicit for completeness. 
Choose two representing spans $(\gamma^1_d,\gamma^1_n)$ and $(\gamma^2_d,\gamma^2_n)$ for $\gamma^1$ and
$\gamma^2$ respectively. Let $(\alpha,m)$ represent $\calP$. To get a representing span for the composition
$\gamma_1\gamma_2$, choose a square filling $(\beta_1,\beta_2)$ of $(\gamma^1_n,\gamma^2_d)$ and take
the span $(\beta_1\gamma^1_d,\beta_2\gamma^2_n)$. This span acts on $(\alpha,m)$ as before and is sketched
diagrammatically as follows:
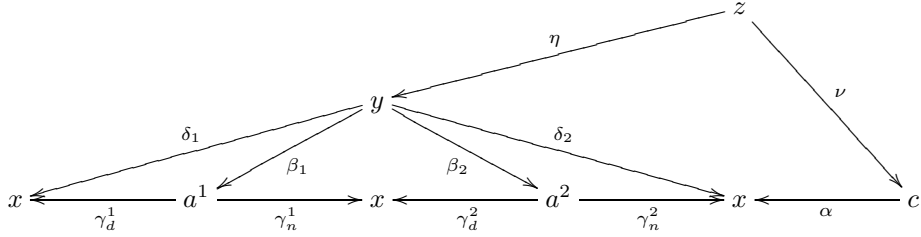
\begin{displaymath}\xymatrix{
	&&&&&&&&z\ar[dllll]_{\eta}\ar[ddrr]^{\nu}&&\\
	&&&&y\ar[dll]^{\beta_1}\ar[drr]_{\beta_2}\ar[dllll]_{\delta_1}\ar[drrrr]^{\delta_2}&&&&&&\\
	x&&a^1\ar[ll]^{\gamma^1_d}\ar[rr]_{\gamma^1_n}&&x&&a^2\ar[ll]^{\gamma^2_d}\ar[rr]_{\gamma^2_n}&&x&&
	c\ar[ll]^\alpha
}\end{displaymath}
So a representative of $(\gamma^1\gamma^2)\cdot\calP$ is given by $(\eta\delta_1,\nu^*(m))$. Now a representative
for $\gamma^2\cdot\calP$ is given by $(\eta\beta_2\gamma^2_d,\nu^*(m))$ because $(\eta\beta_2,\nu)$ is a square
filling for $(\gamma^2_n,\alpha)$. Since $(\eta\beta_1,\id_z)$ is a square filling for
$(\gamma^1_n,\eta\beta_2\gamma^2_d)$, we obtain that $(\eta\beta_1\gamma^1_d,\id_z^*(\nu^*(m)))$ is
a representative of $\gamma^1\cdot(\gamma^2\cdot\calP)$. But this last marked arrow is equal to
$(\eta\delta_1,\nu^*(m))$, q.e.d.

\begin{rem}\label{12397}
	It is not hard to see that $\calP\subset\calQ$ implies $\gamma\cdot\calP\subset\gamma\cdot\calQ$.
\end{rem}

\begin{rem}\label{72125}
	The submultiballs of $\gamma\cdot\calP$ are the multiballs $\gamma\cdot P$ with $P\in\calP$.
\end{rem}

\subsection{Pointwise stabilizers of partitions}

Let $\calP$ be a partition over $x$. By the pointwise stabilizer of $\calP$ we mean
the subgroup
\[\Lambda:=\{\gamma\in\pi_1(\calS,x)\mid\gamma\cdot P=P\text{ for all }P\in\calP\}\]
Fix some representative $(\alpha,m)$ of $\calP$. We can assume without loss of generality that the marking
$m$ on the domain $c$ of $\alpha$ is ordered. That means that if $f\co I\rightarrow S$ is the marking function
of $m$ and whenever $i\into f=j\into f$ for $i<j$, then also $i\into f=k\into f=j\into f$ for every $k$ with
$i<k<j$. This is true in the planar case by definition. In the symmteric case, we can choose a 
colored permutation $\sigma\in\mathfrak{Sym}(C)$ with $\sigma^*(m)$ ordered and replace $(\alpha,m)$ by the equivalent 
marked arrow $(\sigma\alpha,\sigma^*(m))$.

\begin{prop}\label{88520}
	Each symbol of the marking $m$ determines a subword of the word $c=\op{dom}(\alpha)$. Denote
	these subwords by $c_1,...,c_k$ and order them such that $c=c_1\otimes...\otimes c_k$.
	Then we have a well-defined isomorphism of groups
	\[\Xi\co \pi_1(\calS,c_1)\times...\times\pi_1(\calS,c_k)\rightarrow\Lambda\]
	which is given by applying the tensor product of paths and then conjugating with the arrow $\alpha$. More
	explicitly, it is given by sending representing spans $p_1,...,p_k$ to the homotopy class represented by the path
	\begin{displaymath}\begin{array}{ccccccccc}
		&&c_1&\xleftarrow{p_1^\prec}&a_1&\xrightarrow{p_1^\succ}&c_1&&\\
		&&\otimes &\otimes &\otimes &\otimes &\otimes &&\\
		x&\xleftarrow{\hspace{5mm}\alpha\hspace{5mm}}{}&\vdots &\vdots &\vdots &
		\vdots &\vdots &\xrightarrow{\hspace{5mm}\alpha\hspace{5mm}}&x\\
		&&\otimes &\otimes &\otimes &\otimes &\otimes &&\\
		&&c_k&\xleftarrow[p_k^\prec]{}&a_k&\xrightarrow[p_k^\succ]{}&c_k&&
	\end{array}\end{displaymath}
	where $p_i^\prec$ is the arrow pointing to the left and $p_i^\succ$ the arrow pointing to the right in the
	span $p_i$.
\end{prop}

\begin{proof}
	It is not hard to see that the map is independent of the chosen
	representing spans $p_i$ and that it is a group homomorphism. Injectivity follows from Lemma \ref{69343}
	and Lemma \ref{56209} below. Before we prove surjectivity, we want to see that the image really lies in the
	subgroup $\Lambda$. We can use the representative $(\alpha,m)$ to extract representatives of submultiballs
	$P\in\calP$. The subwords $c_i$ are in one to one correspondence with the submultiballs $P\in\calP$. 
	A representative $(\alpha,m_i)$ of $P\in\calP$ corresponding to $c_i$ is obtained from $(\alpha,m)$ 
	by removing all markings except the markings on the subword $c_i$. The representing span of $\Xi(p_1,...,p_k)$ 
	pictured above can be written as $(p^\prec\alpha,p^\succ\alpha)$ where
	$p^\prec=p_1^\prec\otimes...\otimes p_k^\prec$ and $p^\succ=p_1^\succ\otimes...\otimes p_k^\succ$.
	Letting this span act on $(\alpha,m_i)$, we can choose $(\id,p^\succ)$
	as a square filling and the resulting representative is $(p^\prec\alpha,{p^\succ}^*(m_i))$. But this is
	equivalent to $(\alpha,m_i)$ because ${p^\prec}^*(m_i)={p^\succ}^*(m_i)$.
	
	Now we prove surjectivity. Let $\gamma\in\Lambda$ which can be represented by a path of the form
	\[x\xleftarrow{\hspace{5mm}\alpha\hspace{5mm}}c\xleftarrow{\hspace{5mm}z^\prec\hspace{5mm}}a
		\xrightarrow{\hspace{5mm}z^\succ\hspace{5mm}}c\xrightarrow{\hspace{5mm}\alpha\hspace{5mm}}x\]
	Observe the representatives $(\alpha,m_i)$ of the submultiballs $P\in\calP$ from above. A representative
	of $\gamma\cdot[\alpha,m_i]$ is given by $(z^\prec\alpha,{z^\succ}^*(m_i))$. So we have 
	$(\alpha,m_i)\sim(z^\prec\alpha,{z^\succ}^*(m_i))$. Of course, $(z^\prec,\id)$ is a square 
	filling of $(\alpha,z^\prec\alpha)$ and thus
	\[{z^\prec}^*(m_i)={z^\succ}^*(m_i)\]
	Now assume for the moment that the operad $\calO$ is planar. Then it follows easily from these equalities that
	the span $(z^\prec,z^\succ)$ splits as a product according to the decomposition $c=c_1\otimes...\otimes c_k$, 
	i.e.~there are $z_i^\prec\co a_i\rightarrow c_i$ and $z_i^\succ\co a_i\rightarrow c_i$ with 
	$z^\prec=z_1^\prec\otimes...\otimes z_k^\prec$ and $z^\succ=z_1^\succ\otimes...\otimes z_k^\succ$.
	By construction, the spans $(z_i^\prec,z_i^\succ)$ give a preimage of $\gamma$ under $\Xi$. If, on the
	other hand, $\calO$ is symmetric, then there is colored permutation $\sigma\in\mathfrak{Sym}(C)$
	such that the span $(\sigma z^\prec,\sigma z^\succ)$, which is homotopic to $(z^\prec,z^\succ)$, splits as 
	above and we can finish the proof also in this case.
\end{proof}

\begin{lem}\label{56209}
	Let $a\xleftarrow{q}b\xrightarrow{p}a$ be a span in $\calS$ which is a tensor product of $k$ spans
	$a_i\xleftarrow{q_i}b_i\xrightarrow{p_i}a_i$ for $i=1,...,k$, i.e.~$q=q_1\otimes....\otimes q_k$ and
	$p=p_1\otimes...\otimes p_k$. Then the span $(q,p)$ is null-homotopic if and only if each $(q_i,p_i)$
	is null-homotopic.
\end{lem}
\begin{proof}
	It is clear that if each $(q_i,p_i)$ is null-homotopic, then $(q,p)$ is null-homotopic. So we prove the converse.
	We can assume without loss of generality that $q_i\neq \id_I\neq p_i$ where $I$ is the monoidal unit in $\calS$,
	i.e.~the empty word. First observe that $p,q$ are parallel arrows and since $\calS$ satisfies the calculus of fractions, 
	they are homotopic if and only if there is an arrow $r\co c\rightarrow b$ with $rq=rp$. Now, by precomposing with
	an arrow in $\mathfrak{Sym}(C)$ if necessary, we can assume that $r$ is an arrow in
	$\calS(\calO_\mathrm{pl})$, i.e.~a tensor product of operations in $\calO$. Observe that in $\calS$
	we have $\alpha_1\otimes...\otimes\alpha_l=\beta_1\otimes...\otimes\beta_m$ for arrows $\alpha_i\neq\id_I\neq\beta_i$ if 
	and only if $l=m$ and $\alpha_i=\beta_i$ for each $i=1,...,l$. Now it follows easily that $r$ gives arrows $r_1,...,r_k$
	such that $r_iq_i=r_ip_i$ for each $i=1,...,k$. Thus, $q_i$ is homotopic to $p_i$ for each $i=1,...,k$.
\end{proof}

\subsection{The poset of partitions}

From now on, fix some base object $x$ which is split and progressive. More generally, in view of Remark \ref{50091}:
\vspace{2mm}
\begin{center}
	\framebox[1.1\width]{Let $y$ be a split object such that $x$ is $y$--progressive.}
\end{center}
\vspace{2mm}
Furthermore, let $n\in\NN$.

Two objects $a,b$ in $\calS$ are called equivalent if they are isomorphic in $\pi_1(\calS)$, i.e.~there is a path
(equivalently, a span) between them in $\calS$. Of course, $\pi_1(\calS,a)\cong\pi_1(\calS,b)$ in this case.

Let $c=(c_1,...,c_k)$ with $c_i\in C$ an object in $\calS$ and $m$ be a uni-marking on $c$, i.e.~there is only
one symbol in $m$. Then $m$ determines another object $\mathfrak{c}(m)$ by deleting all $c_i$'s which are not marked
by $m$. If $\alpha\co a\rightarrow c$ is an arrow, then $\mathfrak{c}(\alpha^*(m))\sim\mathfrak{c}(m)$ in the above
sense.

Let $B$ be a multiball. If $(\alpha,m)$ and $(\alpha',m')$ are representatives, then 
$\mathfrak{c}(m)\sim\mathfrak{c}(m')$. Thus, each multiball $B$ gives an equivalence class
$\mathfrak{cc}(B)$ of objects.

We say that a partition $\calP$ (over $x$) satisfies the $n$--condition with respect to $y$ if at least $n$ 
submultiballs $P\in\calP$ satisfy $y\in\mathfrak{cc}(P)$. The $n$--condition with respect to $y$ is preserved by 
the action of $\Gamma=\pi_1(\calS,x)$ on the partitions: If $\calP$ satisfies the $n$--condition with respect
to $y$, then also $\gamma\cdot\calP$ satisfies it.

We define a poset $(\mathbb{P},\leq)$: The objects of $\mathbb{P}$ are partitions over $x$ and $\calP\leq\calQ$
if and only if $\calP\supset\calQ$. The group $\Gamma=\pi_1(\calS,x)$ acts on this poset via the action on partitions.
Because of Remark \ref{12397}, the action indeed respects the relation $\leq$.

Since the $n$--condition with respect to $y$ is invariant under the action of $\Gamma$, we can define the invariant 
subposet $(\mathbb{P}_n,\leq)$ to be the full subpost consisting of partitions satisfying the $n$--condition with
respect to $y$. Next, we want
to show that
\begin{itemize}
	\item[{\bf 1.}] $\mathbb{P}_n\neq\emptyset$ and
	\item[{\bf 2.}] $(\mathbb{P}_n,\leq)$ is filtered.
\end{itemize}
This implies that the poset $\mathbb{P}_n$ is \emph{contractible}.

{\bf 1.} Since $x$ is $y$--progressive, there is an arrow $a_1\otimes y\otimes a_2\rightarrow x$. Apply $y$'s split condition 
$(n-1)$ times to find an arrow $z\rightarrow x$ where $z$ has a tensor product decomposition with at least $n$ factors equal to 
$y$. Mark each of these factors with a different symbol and the rest with yet another symbol. This yields a partition 
$\calP\in\mathbb{P}_n$.

{\bf 2.} Let $\calP,\calQ\in\mathbb{P}_n$. We have to find $\calR\in\mathbb{P}_n$ with $\calP,\calQ\leq\calR$.
First we find one in $\mathbb{P}$. Let $(\alpha_\calP,m_\calP)$ and $(\alpha_\calQ,m_\calQ)$ be representatives of 
$\calP$ and $\calQ$ respectively. Choose a square filling $(\beta_\calP,\beta_\calQ)$ of 
$(\alpha_\calP,\alpha_\calQ)$ and set $\delta=\beta_\calP\alpha_\calP=\beta_\calQ\alpha_\calQ$. Now find a full marking
$\mu\subset\beta_\calP^*(m_\calP),\beta_\calQ^*(m_\calQ)$, for example by marking each coordinate of 
$\op{dom}(\delta)$ with a different symbol. Then $\calR=[\delta,\mu]$ is a common refinement of $\calP$
and $\calQ$. Now use that $x$ is $y$--progressive to find an arrow $\eta\co z\rightarrow\op{dom}(\delta)$
where $z$ has a tensor product decomposition with at least one factor equal to $y$. Then apply $y$'s split condition
$(n-1)$ times to obtain an arrow $\nu\co w\rightarrow z$ where $w$ has a tensor product decomposition with at least
$n$ factors equal to $y$. Observe the marked arrow $(\nu\eta\delta,(\nu\eta)^*(\mu))$. The so-called link condition
in Remark \ref{50091} ensures that the $n$ factors of $w$ equal to $y$ are marked with the same symbol 
in the marking $(\nu\eta)^*(\mu)$. Refine this marking such that these factors are marked with 
new different symbols. This gives a representative of a partition satisfying the $n$--condition with respect to $y$, 
refining $\calR$ and thus refining both $\calP$ and $\calQ$.

\vspace{2mm}

A simplex $\sigma$ in the poset $\mathbb{P}_n$ is a finite ascending sequence of objects, written 
$[\calP_0<\calP_1<...<\calP_p]$. We now observe the stabilizer subgroup $\Gamma_\sigma$ of such a simplex.
By definition, an element $\gamma$ is in this stabilizer subgroup if and only if $\{\calP_0,...,\calP_p\}=
\{\gamma\cdot\calP_0,...,\gamma\cdot\calP_p\}$. But since the action of $\gamma$ respects $\leq$, this is
equivalent to $\gamma\cdot\calP_i=\calP_i$ for each $i=0,...,p$. So each $\gamma\in\Gamma_\sigma$ fixes
$\sigma$ vertex-wise. Observe the subgroup
\[\Lambda_\sigma:=\{\gamma\in\Gamma\mid\gamma\cdot P=P\text{ for all }P\in\calP_p\}<\Gamma\]
By Proposition \ref{88520}, we know that $\Lambda_\sigma\cong\pi_1(\calS,c_1)\times...\times\pi_1(\calS,c_k)$
for appropriate objects $c_i$. Since $\calP_p$ satisfies the $n$--condition with respect to $y$, at least $n$ 
of these objects are equivalent to $y$ and thus at least $n$ of the factors in the product decomposition of
$\Lambda_\sigma$ are isomorphic to $\Upsilon:=\pi_1(\calS,y)$. So we find a normal subgroup 
$\Lambda'_\sigma\vartriangleleft\Lambda_\sigma$ with $\Lambda'_\sigma\cong\Upsilon^n$. Below, we will show that 
$\Lambda_\sigma$ is a normal subgroup of $\Gamma_\sigma$. So we arrive at the following situation
\vspace{2mm}
\begin{center}
	\framebox[1.1\width]{$\Upsilon^n\cong\Lambda'_\sigma\vartriangleleft\Lambda_\sigma\vartriangleleft\Gamma_\sigma$}
\end{center}
\vspace{2mm}

\begin{lem}
	Let $\calR_1,\calR_2$ be semi-partitions and $\calP$ be a partition with $\calP\subset\calR_1$.
	Assume
	\[\forall_{R_1\in\calR_1}\ \exists_{R_2\in\calR_2}\ \forall_{P\in\calP}\ P\subset R_1
		\Longrightarrow P\subset R_2\]
	Then we have $\calR_1\subset\calR_2$.
\end{lem}
\begin{proof}
	By applying the square filling technique twice, we find an arrow $\delta$ with three markings 
	$m_\calP,m_{\calR_1},m_{\calR_2}$ on its domain such that $(\delta,m_\calP)$ represents $\calP$ and 
	$(\delta,m_{\calR_i})$ represents $\calR_i$. Since $\calP\subset\calR_1$ we have $(\delta,m_\calP)\subset
	(\delta,m_{\calR_1})$ and therefore $m_\calP\subset m_{\calR_1}$. Note that $\calP$ is a partition and 
	therefore $m_\calP$ is a full marking. Now the assumption of the statement implies 
	$m_{\calR_1}\subset m_{\calR_2}$ and thus $\calR_1\subset\calR_2$.
\end{proof}

We first show that $\Lambda_\sigma$ is contained in $\Gamma_\sigma$. So let $\gamma\in\Lambda_\sigma$,
i.e.~$\gamma\cdot P=P$ for all $P\in\calP_p$. In particular, we have $\gamma\cdot\calP_p=\calP_p$ (Proposition \ref{45613}
and Remark \ref{72125}).
We have to show $\gamma\cdot\calP_i=\calP_i$ also for the other $i$'s. Write $\calP:=\calP_p$ and $\calR:=\calP_i$ for some
other $i$. Then we have $\calP\subset\calR$. We want to apply the above lemma to $\calR_1=\calR$ and
$\calR_2=\gamma\cdot\calR$ and deduce $\calR\subset\gamma\cdot\calR$. So let $R\in\calR$ and observe
$\gamma\cdot R\in\gamma\cdot\calR$. Let $P\in\calP$ with $P\subset R$. Then $P=\gamma\cdot P\subset\gamma\cdot R$
and the assumption of the lemma is satisfied. Similarly, we get $\calR\subset\gamma^{-1}\cdot\calR$ and thus
$\gamma\cdot\calR\subset\calR$. This yields $\gamma\cdot\calR=\calR$, q.e.d.

Now we show that $\Lambda_\sigma$ is normal in $\Gamma_\sigma$. Let $\gamma\in\Gamma_\sigma$ and
$\alpha\in\Lambda_\sigma$. We have to show $\gamma^{-1}\alpha\gamma\in\Lambda_\sigma$, 
i.e.~$\gamma^{-1}\alpha\gamma\cdot P=P$ for all $P\in\calP_p=:\calP$ or equivalently 
$\alpha\cdot(\gamma\cdot P)=\gamma\cdot P$ for all $P\in\calP$. Since $\gamma\cdot\calP=\calP$, 
we have a bijection $f\co \{P\in\calP\}\rightarrow\{P\in\calP\}$ such that $\gamma\cdot P=P\into f$ for all
$P\in\calP$ (Proposition \ref{45613}). Consequently, if $P\in\calP$, 
$\alpha\cdot(\gamma\cdot P)=\alpha\cdot(P\into f)=P\into f=\gamma\cdot P$, q.e.d.

\subsection{End of the proof}

Let $\mathfrak{P}$ be a property of groups which is closed under taking products and let $\calM$ be a coefficient
system which is $\mathfrak{P}$--K\"unneth and $\mathfrak{P}$--inductive.
We will only give the proof for homology. Using analogous devices for cohomology, we obtain a proof of the 
cohomological version of the statement.

Our main tool will be a spectral sequence explained in Brown's book \cite{bro:fpo}*{Chapter VII.7} (see also
\cite{sa-th:laa}*{Subsection 4.1}). If we plug in our $\Gamma$--complex $(\mathbb{P}_n,\leq)$ and the $\ZZ\Gamma$--module $\calM\Gamma$,
we obtain a spectral sequence $E^k_{pq}$ with
\[E^1_{pq}=\bigoplus_{\sigma\in\Sigma_p}H_q(\Gamma_\sigma,\calM\Gamma)\Rightarrow 
	H_{p+q}(\Gamma,\calM\Gamma)\]
where $\Sigma_p$ is set of $p$--cells representing the $\Gamma$--orbits of $(\mathbb{P}_n,\leq)$. This uses that the poset $\mathbb{P}_n$
is contractible and that the cell stabilizers fix the cells pointwise.

We assumed that $\Upsilon$ satisfies $\mathfrak{P}$ and that $H_0(\Upsilon,\calM\Upsilon)=0$. Applying the
$\mathfrak{P}$--K\"unneth property $(n-1)$ times, we obtain $H_k(\Upsilon^n,\calM\Upsilon^n)=0$ for $k\leq n-1$.
So we have $H_k(\Lambda'_\sigma,\calM\Lambda'_\sigma)=0$ for $k\leq n-1$. The property $\mathfrak{P}$--inductive
yields $H_k(\Lambda'_\sigma,\calM\Gamma)=0$ for $k\leq n-1$. Since $\Lambda'_\sigma\vartriangleleft
\Lambda_\sigma\vartriangleleft\Gamma_\sigma$, we can apply the Hochschild--Serre spectral sequence twice to
obtain $H_k(\Gamma_\sigma,\calM\Gamma)=0$ for $k\leq n-1$. The above spectral sequence now yields
$H_k(\Gamma,\calM\Gamma)=0$ for $k\leq n-1$. Since $n$ was arbitrary, the result follows.

\section{Non-amenability and infiniteness}

In this section we use the techniques from Section \ref{96238} to prove non-amenability and infiniteness of some operad groups.
Note that semi-partitions and the action on the set of semi-partitions can also be defined in the braided case.

\begin{lem}\label{67656}
	If $\calO$ satisfies the calculus of fractions, then the action of the colored permutations in
	$\op{Aut}_{\mathfrak{Sym}(C)}(X)$ or the colored braids in 
	$\op{Aut}_{\mathfrak{Braid}(C)}(X)$ on the set of arrows $\op{Hom}_{\calS(\calO)}(X,Y)$
	is free. In particular, in the operad $\calO$, the action of the symmetric groups or the braid
	groups on the sets of operations is free.
\end{lem}
\begin{proof}
	Let $[\alpha,\Theta]$ be an element in $\op{Hom}_{\calS(\calO)}(X,Y)$ and $\sigma\in\op{Aut}_{\mathfrak{Sym}(C)}(X)$ 
	or $\sigma\in\op{Aut}_{\mathfrak{Braid}(C)}(X)$. We have to show that $[\sigma,\id]*[\alpha,\Theta]=[\alpha,\Theta]$ 
	implies that $\sigma$ is trivial. From this equality and the equalization property of $\calS(\calO)$, we obtain an arrow 
	$z:=[\delta,\Psi]$ with $z*[\sigma,\id]=z$. We can assume without loss of generality that $\delta=\id$.
	We then have $z*[\sigma,\id]=[\bar{\sigma},\bar{\Psi}]$ with $\bar{\sigma}=\Psi{\curvearrowright}\sigma$ and 
	$\bar{\Psi}=\Psi{\curvearrowleft}\sigma$. Consequently, the pairs $(\bar{\sigma},\bar{\Psi})$ and $(\id,\Psi)$ 
	are equivalent in $\mathfrak{Sym}(C)\times\calS(\calO_\mathrm{pl})$ or $\mathfrak{Braid}(C)\times\calS(\calO_\mathrm{pl})$.
	This is only possible if $\sigma$ is trivial.
\end{proof}

Let $\calO$ be a symmetric or braided operad. Let $\alpha$ be an arrow in $\calS(\calO)$. 
For any colored permutation $\sigma\in\mathfrak{Sym}(C)$ or colored braid $\sigma\in\mathfrak{Braid}(C)$ with suitable domain and 
codomain, we can form the group element $\gamma$ represented by the span $(\alpha,[\sigma,\id]*\alpha)$. Recall that the first 
arrow always denotes the denominator, i.e.~points to the left.

\begin{lem}\label{22286}
	Assume $\calO$ satisfies the calculus of fractions. Then
	\[\sigma\neq 1\ \ \Longrightarrow\ \ \gamma\neq 1\]
\end{lem}
\begin{proof}
	First consider the symmetric case. Observe the semi-partition $\calR$ represented by the marked arrow 
	$(\alpha,m)$ where $m$ is a marking on the domain of $\alpha$ with 
	only one marked coordinate and this coordinate is non-trivially permuted by $\sigma^{-1}$. 
	It is easy to see that $\gamma\cdot\calR$ is represented by $(\alpha,[\sigma,\id]^*(m))$. 
	The marking $m':=[\sigma,\id]^*(m)$ is different from $m$ because $\sigma^{-1}$ maps the only
	marked coordinate of $m$ to a different coordinate by assumption. 
	From Remark \ref{00176} it follows that the marked
	arrow $(\alpha,m)$ cannot be equivalent to $(\alpha,m')$ and thus $\gamma\cdot\calR\neq\calR$. Consequently
	$\gamma\neq 1$.
	
	Now if $\calO$ is braided we can apply the above argument verbatim if we require that the braid $\sigma$
	has a non-trivial permutation part. But there are of course non-trivial braids which are trivial as
	permutations (so-called pure braids). 
	Assume that $\sigma$ is such a pure braid and that $\gamma=1$. Note that the latter means that
	the parallel arrows $\sigma\alpha:=[\sigma,\id]*\alpha$ and $\alpha$ are homotopic. 
	Since $\calS(\calO)$ satisfies the calculus of fractions, this means that there is an arrow $\delta$ with
	$\delta*\sigma\alpha=\delta*\alpha$. We can assume without loss of generality that 
	$\delta\in\calS(\calO_{\mathrm{pl}})$, i.e.~that $\delta=[\id,\Theta]$. Composing in
	$\calS(\calO)$, we get $\delta*[\sigma,\id]=[\bar{\sigma},\bar{\Theta}]$ where
	$\bar{\sigma}=\Theta{\curvearrowright}\sigma$ and $\bar{\Theta}=\Theta{\curvearrowleft}\sigma$. 
	Using that $\sigma$ is pure we immediately see $\bar{\Theta}=\Theta$. So we have
	\[[\bar{\sigma},\id]*\big([\id,\Theta]*\alpha\big)=\big(\delta*[\sigma,\id]\big)*\alpha=[\id,\Theta]*\alpha\]
	Lemma \ref{67656} now gives $\bar{\sigma}=1$ and thus $\sigma=1$.
\end{proof}

Denoting the element $\gamma$ suggestively by $\xleftarrow{\alpha}\sigma\xrightarrow{\alpha}$, the lemma implies
that two such group elements $\xleftarrow{\alpha}\sigma\xrightarrow{\alpha}$ and 
$\xleftarrow{\alpha}\sigma'\xrightarrow{\alpha}$ are equal if and only if $\sigma=\sigma'$. We will use this now
to give a proof of the following proposition.

\begin{prop}\label{22947}
	Let $\calO$ be a symmetric operad satisfying the calculus of fractions and let $X$ be a split 
	object of $\calS(\calO)$. Then $\Gamma=\pi_1(\calO,X)$ contains a non-abelian free subgroup and
	is therefore non-amenable.
\end{prop}
\begin{proof}	
	Using the split condition on $X$, we will explicitly construct two non-trivial elements
	$\gamma_1,\gamma_2\in\Gamma$ of order $2$ and $3$ respectively. Then we will define two disjoint subsets
	$A_1,A_2$ of the set of semi-partitions over $X$ such that $\gamma_1\cdot A_2\subset A_1$ and 
	$\gamma_2^n\cdot A_1\subset A_2$ for $n=1,2$. The Ping--Pong Lemma then shows that the subgroup 
	$\langle\gamma_1,\gamma_2\rangle$ generated by the two elements $\gamma_1$ and $\gamma_2$ is isomorphic 
	to the free product $\langle\gamma_1\rangle*\langle\gamma_2\rangle$. So we have found a subgroup
	which is isomorphic to $\ZZ_2*\ZZ_3$. Since $\ZZ_2*\ZZ_3$ contains a free non-abelian subgroup, the
	proof of the proposition is then complete.
	
	We now give the constructions. Because $X$ is split, there is an arrow
	\[\varpi\co A_1\otimes X\otimes A_2\otimes X\otimes A_3\rightarrow X\]
	For better readability, we assume that $X$ is a single color and $A_1=A_2=A_3=I$. The construction 
	goes the same way in the general case (with obvious modifications). So we assume that $\varpi$ is just
	an operation with two inputs of color $X$ and an output of color $X$. Define $\gamma_1$ to be the
	following element
	
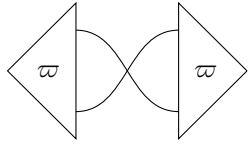
\begin{center}\begin{tikzpicture}[scale=0.45]
		\draw (0,0) -- +(2,2) -- +(2,-2) -- +(0,0);
		\draw (7,0) -- +(-2,2) -- +(-2,-2) -- +(0,0);
		\draw (2,1.2) to[out=right,in=left] (5,-1.2);
		\draw (2,-1.2) to[out=right,in=left] (5,1.2);
		\node at (1.2,0) {$\varpi$};
		\node at (5.8,0) {$\varpi$};
	\end{tikzpicture}\end{center}
	and $\gamma_2$ to be the following element
	
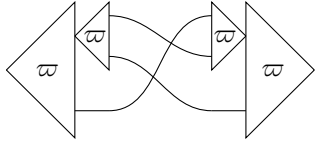
\begin{center}\begin{tikzpicture}[scale=0.45]
		\draw (0,0) -- +(2,2) -- +(2,-2) -- +(0,0);
		\draw (2,1) -- +(1,1) -- +(1,-1) -- +(0,0);
		\draw (2,-1.2) to (3,-1.2);
		\draw (3,1.6) to[out=right,in=left] (6,0.4);
		\draw (3,0.4) to[out=right,in=left] (6,-1.2);
		\draw (9,0) -- +(-2,2) -- +(-2,-2) -- +(0,0);
		\draw (7,1) -- +(-1,1) -- +(-1,-1) -- +(0,0);
		\draw (7,-1.2) to (6,-1.2);
		\draw (3,-1.2) to[out=right,in=left] (6,1.6);
		\node at (1.2,0) {$\varpi$};
		\node at (7.8,0) {$\varpi$};
		\node at (2.6,1) {$\varpi$};
		\node at (6.4,1) {$\varpi$};
	\end{tikzpicture}\end{center}
	Lemma \ref{22286} implies that $\gamma_1$ is of order $2$ and $\gamma_2$ is of order $3$. Let $B_1$
	be the ball represented by the marked arrow
	
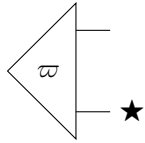
\begin{center}\begin{tikzpicture}[scale=0.45]
		\draw (0,0) -- +(2,2) -- +(2,-2) -- +(0,0);
		\draw (2,1.2) to (3,1.2);
		\draw (2,-1.2) to (3,-1.2) node[right]{$\bigstar$};
		\node at (1.2,0) {$\varpi$};
	\end{tikzpicture}\end{center}
	and let $B_2$ be the ball represented by the marked arrow
	
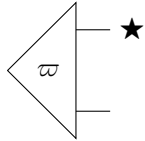
\begin{center}\begin{tikzpicture}[scale=0.45]
		\draw (0,0) -- +(2,2) -- +(2,-2) -- +(0,0);
		\draw (2,1.2) to (3,1.2) node[right]{$\bigstar$};
		\draw (2,-1.2) to (3,-1.2);
		\node at (1.2,0) {$\varpi$};
	\end{tikzpicture}\end{center}
	Composing the operation $\varpi$ several times, one gets operations that look like binary trees. Call them
	$\varpi$--tree operations. Now define $A_1$ to be the set of all balls $B\subset B_1$ which are represented
	by $\varpi$--tree operations. Similarly, define $A_2$ to be the set of all balls $B\subset B_2$ which are
	represented by $\varpi$--tree operations. For example, the following marked arrows represent balls in $A_1$
	
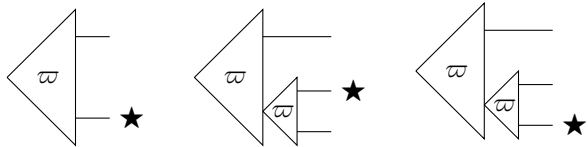
\begin{center}
	\begin{tikzpicture}[scale=0.45]
		\draw (0,0) -- +(2,2) -- +(2,-2) -- +(0,0);
		\draw (2,1.2) to (3,1.2);
		\draw (2,-1.2) to (3,-1.2) node[right]{$\bigstar$};
		\node at (1.2,0) {$\varpi$};
	\end{tikzpicture}
	\hspace{3mm}
	\begin{tikzpicture}[scale=0.45]
		\draw (0,0) -- +(2,2) -- +(2,-2) -- +(0,0);
		\draw (2,-1) -- +(1,1) -- +(1,-1) -- +(0,0);
		\draw (2,1.2) to (4,1.2);
		\draw (3,-1.6) to (4,-1.6);
		\draw (3,-0.4) to (4,-0.4) node[right]{$\bigstar$};
		\node at (1.2,0) {$\varpi$};
		\node at (2.6,-1) {$\varpi$};
	\end{tikzpicture}
	\hspace{3mm}
	\begin{tikzpicture}[scale=0.45]
		\draw (0,0) -- +(2,2) -- +(2,-2) -- +(0,0);
		\draw (2,-1) -- +(1,1) -- +(1,-1) -- +(0,0);
		\draw (2,1.2) to (4,1.2);
		\draw (3,-1.6) to (4,-1.6) node[right]{$\bigstar$};
		\draw (3,-0.4) to (4,-0.4);
		\node at (1.2,0) {$\varpi$};
		\node at (2.6,-1) {$\varpi$};
	\end{tikzpicture}
	\end{center}
	and the following marked arrows represent balls in $A_2$
	
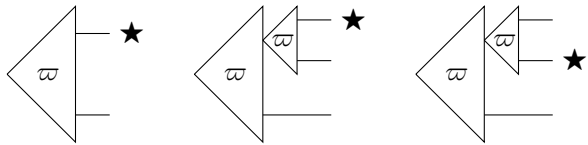
\begin{center}
	\begin{tikzpicture}[scale=0.45]
		\draw (0,0) -- +(2,2) -- +(2,-2) -- +(0,0);
		\draw (2,1.2) to (3,1.2) node[right]{$\bigstar$};
		\draw (2,-1.2) to (3,-1.2);
		\node at (1.2,0) {$\varpi$};
	\end{tikzpicture}
	\hspace{3mm}
	\begin{tikzpicture}[scale=0.45]
		\draw (0,0) -- +(2,2) -- +(2,-2) -- +(0,0);
		\draw (2,1) -- +(1,1) -- +(1,-1) -- +(0,0);
		\draw (2,-1.2) to (4,-1.2);
		\draw (3,1.6) to (4,1.6) node[right]{$\bigstar$};
		\draw (3,0.4) to (4,0.4);
		\node at (1.2,0) {$\varpi$};
		\node at (2.6,1) {$\varpi$};
	\end{tikzpicture}
	\hspace{3mm}
	\begin{tikzpicture}[scale=0.45]
		\draw (0,0) -- +(2,2) -- +(2,-2) -- +(0,0);
		\draw (2,1) -- +(1,1) -- +(1,-1) -- +(0,0);
		\draw (2,-1.2) to (4,-1.2);
		\draw (3,1.6) to (4,1.6);
		\draw (3,0.4) to (4,0.4) node[right]{$\bigstar$};
		\node at (1.2,0) {$\varpi$};
		\node at (2.6,1) {$\varpi$};
	\end{tikzpicture}
	\end{center}
	It is straightforward to check $\gamma_1\cdot A_2\subset A_1$ and $\gamma_2\cdot A_1\subset A_2$ and 
	$\gamma_2^2\cdot A_1\subset A_2$, so the proof is completed.
\end{proof}

Next we give sufficient conditions for infiniteness of operad groups.

\begin{prop}\label{96213}
	Let $\calO$ be a planar, symmetric or braided operad satisfying the calculus of fractions and let
	$X$ be a split object of $\calS(\calO)$. Then $\Gamma:=\pi_1(\calO,X)$ contains an infinite
	cyclic subgroup and is therefore infinite.
\end{prop}
\begin{proof}
	Because $X$ is split, there is an arrow
	\[\varpi\co A_1\otimes X\otimes A_2\otimes X\otimes A_3\rightarrow X\]
	For better readability, we assume that $X$ is a single color and $A_1=A_2=A_3=I$. The construction 
	goes the same way in the general case (with obvious modifications). So we assume that $\varpi$ is just
	an operation with two inputs of color $X$ and an output of color $X$. Define $\gamma$ to be the
	following element
	
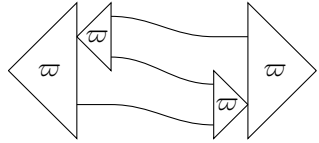
\begin{center}\begin{tikzpicture}[scale=0.45]
		\draw (0,0) -- +(2,2) -- +(2,-2) -- +(0,0);
		\draw (2,1) -- +(1,1) -- +(1,-1) -- +(0,0);
		\draw (9,0) -- +(-2,2) -- +(-2,-2) -- +(0,0);
		\draw (7,-1) -- +(-1,1) -- +(-1,-1) -- +(0,0);
		\draw (3,1.6) to[out=right,in=left] (6,1) to (7,1);
		\draw (3,0.4) to[out=right,in=left] (6,-0.4);
		\draw (2,-1) to (3,-1) to[out=right,in=left] (6,-1.6);
		\node at (1.2,0) {$\varpi$};
		\node at (7.8,0) {$\varpi$};
		\node at (6.4,-1) {$\varpi$};
		\node at (2.6,1) {$\varpi$};
	\end{tikzpicture}\end{center}
	Formally, $\gamma$ is represented by the span
	$\big((\varpi,\id)*\varpi,(\id,\varpi)*\varpi\big)$. We claim that $\gamma$ has 
	infinite order. The element $\gamma^n$ is represented by the span (for better readability, we use the same symbol
	$\id$ for different identities)
	\[\big((\varpi,\id)*...*(\varpi,\id)*\varpi,
		(\id,\varpi)*...*(\id,\varpi)*\varpi\big)\]
	By Lemma \ref{69343}, this span is null-homotopic if and only of the span (remove the last $\varpi$ in both arrows)
	\[\big((\varpi,\id)*...*(\varpi,\id),
		(\id,\varpi)*...*(\id,\varpi)\big)=:(\varpi_1,\varpi_2)\]
	is null-homotopic. This is true if and only if there is an arrow $r$ with $r\varpi_1=r\varpi_2$. The arrow
	$r$ can be chosen to lie in $\calS(\calO_{\mathrm{pl}})$. But note that $\varpi_1$ splits as
	\[\varpi_1=\big((\varpi,\id)*...*(\varpi,\id)*\varpi\big)\otimes\id_X\]
	and $\varpi_2$ splits as
	\[\varpi_2=\id_X\otimes\big((\id,\varpi)*...*(\id,\varpi)*\varpi\big)\]
	It can easily be seen that such an arrow $r$ cannot exist because otherwise operations with a different number
	of inputs must be equal. Consequently, each $\gamma^n$ is non-trivial and therefore $\gamma$ has infinite order.
\end{proof}

\section{Applications}

Observe the $1$--dimensional planar cube cutting operads and the $d$--dimensional symmetric cube cutting operads from 
\cite{thu:ogatfp}*{Subsection 3.5}. They all satisfy the (cancellative) calculus of fractions.
Furthermore, they are monochromatic and possess operations of arbitrarily high degree. From Remarks \ref{92437} and 
\ref{71138} it follows that all objects (except the uninteresting unit object) are split and progressive. So Theorem
\ref{20869} is applicable to these operads. Furthermore, the corresponding operad groups are all infinite by Proposition
\ref{96213} and non-amenable in the symmetric case by Proposition \ref{22947}.

\vspace{2mm}

Observe now the local similarity operads. Let $\mathrm{Sim}_X$ be a finite similarity structure on the compact ultrametric space $X$. 
When choosing a ball in each $\mathrm{Sim}_X$--equivalence of balls, we obtain a symmetric operad with transformations $\calO$.
The colors of $\calO$ are the chosen balls. We choose $X$ for the $\mathrm{Sim}_X$--equivalence class $[X]$. We already know
that $\calO$ satisfies the (cancellative) calculus of fractions. In \cite{sa-th:laa}*{Definition 3.1} we called $\mathrm{Sim}_X$ dually 
contracting if there are two disjoint proper subballs $B_1,B_2$ of $X$ together with similarities $X\rightarrow B_i$ in $\mathrm{Sim}_X$. 
This easily implies that $X$ is split.
\begin{lem}
	The color $X$ is progressive provided $\mathrm{Sim}_X$ is dually contracting.
\end{lem}
\begin{proof}
	Let $\theta=(f_1,...,f_l)$ be an operation with output $X$. This means that the $f_i\co B_i\rightarrow X$ are 
	$\mathrm{Sim}_X$--embeddings (i.e.~$f_i$ yields a similarity in $\mathrm{Sim}_X$ when the codomain is restricted to 
	the image) such that the images of the $f_i$ are pairwise disjoint and their union is $X$. So the images
	$\op{im}(f_i)$ form a partition $\calP$ of $X$ into balls. If we apply \cite{sa-th:laa}*{Lemma 3.7} to this
	partition, we find a $j$ and a small ball $B$ which is $\mathrm{Sim}_X$--equivalent to $X$ and such that
	$B\subset\op{im}(f_j)$. Using this, we can construct an operation $\psi=(g_1,...,g_k)$ with codomain $B_j$ such
	that $g_1\co X\rightarrow B_j$. From Remark \ref{71138} it now follows that $X$ is progressive.
\end{proof}
Consequently, Theorem \ref{20869} is applicable to dually contracting local similarity operads. Furthermore, the corresponding operad 
groups based at $X$ are all infinite by Proposition \ref{96213} and non-amenable by Proposition \ref{22947}.

\subsection{$L^2$--homology}\label{72248}

For a group $G$, let $l^2G$ be the Hilbert space with Hilbert base $G$. Thus, elements in $l^2G$ are formal sums
$\sum_{g\in G}\lambda_g g$ with $\lambda_g\in\CC$ such that $\sum_{g\in G}|\lambda_g|^2<\infty$. Left multiplication with
elements in $G$ induces an isometric $G$--action on $l^2G$. Denote the set of $G$--equivariant linear bounded operators
$l^2G\rightarrow l^2G$ by $\calB^G(l^2G)$, a subalgebra of the algebra of all bounded linear operators $\calB(l^2G)$.
\emph{Right} multiplication with an element $\gamma\in G$ induces a $G$--equivariant linear bounded operator
$\gamma\into\rho\co l^2G\rightarrow l^2G$. This induces a homomorphism $\rho\co\CC G\rightarrow\calB(l^2G)$ from the group ring
into the algebra of bounded linear operators, i.e.~$1\into\rho=\id$ and $(\gamma_1\gamma_2)\into\rho=
(\gamma_1\into\rho)*(\gamma_1\into\rho)$. The closure of the image of this map with respect to the weak or strong
operator topology is called the von Neumann algebra $\calN G$ associated to $G$. It is equal to the subalgebra of all
$G$--equivariant bounded linear operators $\calB^G(l^2G)\subset\calB(l^2G)$ \cite{lue:lta}*{Example 9.7}.

We will cite some known results about this von Neumann algebra in order to deduce a corollary for $l^2$--homology.
\begin{itemize}
	\item ({\it $\calN$ is inductive}) Let $H$ be a subgroup of $G$ and $A\in\calB^H(l^2H)$. Then 
	$\CC G\otimes_{\CC H}l^2H\subset l^2G$ is a dense $G$--invariant subspace and
	\[\id_{\CC G}\otimes_{\CC H}A\co \CC G\otimes_{\CC H}l^2H\rightarrow \CC G\otimes_{\CC H}l^2H\]
	is a $G$--equivariant linear map which is bounded with respect to the norm coming from $l^2G$. Consequently, 
	it can be extended to an element in $\calB^G(l^2G)$. We obtain a map $\calN H\rightarrow\calN G$ which is an 
	injective ring homomorphism. So if $H<G$, then $\calN H$ is a subring of $\calN G$. Even more is true: 
	It is a faithfully flat ring extension \cite{lue:lta}*{Theorem 6.29}. From this, it
	follows easily that the coefficient system $\calN$ is inductive.
	\item ({\it $\calN$ is K\"unneth}) If $H_1,H_2$ are two subgroups of $G$ which commute in $G$, 
	i.e.~$h_1h_2=h_2h_1$ for all $h_1\in H_1$ and $h_2\in H_2$, then $\calN H_1$ and $\calN H_2$ commute in 
	$\calN G$. In particular, $\calN H_1\otimes_\CC\calN H_2$ is a subring of $\calN G$. This implies, using a 
	standard homological algebraic argument \cite{lue:lta}*{Lemma 12.11(3)}, that $\calN$ is K\"unneth.
	\item ({\it $H_0$ and amenability}) Going back to a result of Kesten, the $0$'th group homology of a group $G$ 
	with coefficients in the von Neumann algebra $\calN G$ vanishes if and only if $G$ is non-amenable
	\cite{lue:lta}*{Lemma 6.36}. So we have
	\[H_0(G,\calN G)=0\hspace{5mm}\Longleftrightarrow\hspace{5mm}G\text{ non-amenable}\]
	\item ({\it Relationship with $l^2$--homology}) From \cite{lue:lta}*{Lemma 6.97} or \cite{lue:lta}*{Theorem 6.24(3)}
	we get for groups $G$ of type $F_{\infty}$ and every $k\geq 0$
	\[H_k(G,\calN G)=0\hspace{5mm}\Longleftrightarrow\hspace{5mm}H_k(G,l^2G)=0\]
\end{itemize}

Applying Theorem \ref{20869} to these observations, we get the following corollary.
\begin{cor}
	Let $\calO$ be a planar or symmetric operad which satisfies the calculus of fractions.
	Let $X$ be a split progressive object of $\calS(\calO)$. Set $\Gamma:=\pi_1(\calO,X)$ and assume
	that $\Gamma$ is non-amenable. Then
	\[H_k(\Gamma,\calN\Gamma)=0\]
	for all $k\geq 0$. If $\Gamma$ is also of type $F_\infty$ (e.g.~if the conditions in \cite{thu:ogatfp}*{Theorem 4.3}
	are satisfied), we also have
	\[H_k(\Gamma,l^2\Gamma)=0\]
	for all $k\geq 0$.
\end{cor}

From Proposition \ref{22947}, we get the following corollary.
\begin{cor}
	Let $\calO$ be a symmetric operad which satisfies the calculus of fractions.
	Let $X$ be a split progressive object of $\calS(\calO)$. Set $\Gamma:=\pi_1(\calO,X)$. Then
	\[H_k(\Gamma,\calN\Gamma)=0\]
	for all $k\geq 0$. If $\Gamma$ is also of type $F_\infty$, we have
	\[H_k(\Gamma,l^2\Gamma)=0\]
	for all $k\geq 0$.
\end{cor}

From the remarks at the beginning of this section and from \cite{thu:ogatfp}*{Subsection 4.6}, we get the following corollary.
\begin{cor}\label{57450}
	Let $\calO$ be a symmetric cube cutting operad or a local similarity operad coming from a dually contracting finite similarity structure
	$\mathrm{Sim}_X$. In the first case, let $A$ be any object in $\calS(\calO)$ different from the monoidal 
	unit $I$. In the second case, let $A$ be the object $X$. Set $\Gamma=\pi_1(\calO,A)$. Then
	\[H_k(\Gamma,\calN\Gamma)=0\]
	for all $k\geq0$. Assume furthermore that $\mathrm{Sim}_X$ is rich in ball contractions \cite{fa-hu:fpo}*{Definition 5.12},
	in other words, the associated operad $\calO$ is color-tame in the sense of \cite{thu:ogatfp}*{Definition 4.2}. Then we also have
	\[H_k(\Gamma,l^2\Gamma)=0\]
	for all $k\geq 0$. 
%	In particular, we have 
%	\[H_k(V,\calN V)=0\ \ \text{and}\ \ H_k(V,l^2V)=0\]
%	for all $k\geq 0$.
\end{cor}

In particular, we obtain that the Higman--Thompson groups $V_{n,r}$ and the higher-dimensional Thompson groups $nV$ (see \cite{bri:hdt})
are $l^2$-invisible. This answers a question posed by L\"uck \cite{lue:lta}*{Remark 12.4}: The \emph{Zero-in-the-spectrum conjecture} by Gromov
says that whenever $M$ is an aspherical closed Riemannian manifold, then there is always a dimension $p$ such that
zero is contained in the spectrum of the minimal closure of the Laplacian acting on smooth $p$--forms on the
universal covering of $M$:
\[\exists_{p\geq 0}\ \ 0\in\op{spec}\big(\op{cl}(\Delta_p)\co D\subset L^2\Omega^p(\widetilde{M})
	\rightarrow L^2\Omega^p(\widetilde{M})\big)\]
By \cite{lue:lta}*{Lemma 12.3}, this is equivalent to
\[\exists_{p\geq 0}\ \ H_p(G,\calN G)\neq 0\]
for $G=\pi_1(M)$. Dropping Poincar\'e duality from the assumptions, we arrive at the following question: If $G$ is a
group of type $F$ (i.e.~there exists a compact classifying space for $G$), then is there a $p$ with 
$H_p(G,\calN G)\neq 0$? Relaxing the assumption on the finiteness property, we arrive at the following question: If $G$ is a
group of type $F_\infty$, then is there a $p$ with $H_p(G,\calN G)\neq 0$? Corollary \ref{57450}
gives explicit counterexamples to this question.

\subsection{Cohomology with coefficients in the group ring}

We want to apply the cohomological version of Theorem \ref{20869} to $\calM G:=\ZZ G$. To
this end, we want to show that $\ZZ G$ is $FP_\infty$--K\"unneth and $FP_\infty$--inductive (in cohomology).
The first property follows from \cite{sa-th:laa}*{Proposition 4.3}. The second property follows from the
observation that $\ZZ G$ is a free $\ZZ H$--module if $H<G$ and that group cohomology of groups of type 
$FP_\infty$ commutes with direct limits in the coefficients \cite{bro:cog}*{Theorem VIII.4.8}.
From Theorem \ref{20869}, Proposition \ref{96213} and $H^0(G,\ZZ G)=(\ZZ G)^G=0$ for infinite $G$, we obtain:
\begin{cor}
	Let $\calO$ be a planar or symmetric operad which satisfies the calculus of fractions.
	Let $X$ be a split progressive object of $\calS(\calO)$. Set $\Gamma:=\pi_1(\calO,X)$ and assume
	that $\Gamma$ is of type $FP_\infty$ (e.g.~if the conditions in Theorem \cite{thu:ogatfp}*{Theorem 4.3} are satisfied). Then
	\[H^k(\Gamma,\ZZ\Gamma)=0\]
	for all $k\geq 0$.
\end{cor}

Recall that type $F_\infty$ implies type $FP_\infty$ and note that, in this case, $H^k(\Gamma,\ZZ\Gamma)=0$ for all $k\geq 0$ implies that $\Gamma$ has 
infinite cohomological dimension \cite{bro:cog}*{Propositions VIII.6.1 and VIII.6.7}. 
Unfortunately, this tells us that none of these groups can be of type $F$ and consequently, we cannot find such a group
which is also $l^2$-invisible.

From the remarks at the beginning of this section and from \cite{thu:ogatfp}*{Subsection 4.6}, we get the following corollary.
\begin{cor}
	Let $\calO$ be a planar or symmetric cube cutting operad or a local similarity operad coming from a dually contracting 
	finite similarity structure $\mathrm{Sim}_X$ which is also rich in ball contractions. In the first two cases, let $A$ be any object in 
	$\calS(\calO)$ different from the monoidal unit $I$. In the last case, let $A$ be the object $X$. Set $\Gamma=\pi_1(\calO,A)$. Then
	\[H^k(\Gamma,\ZZ\Gamma)=0\]
	for all $k\geq 0$.
\end{cor}

In particular, we obtain $H^k(F,\ZZ F)=0$ and $H^k(V,\ZZ V)=0$ for all $k\geq 0$. This has been shown before in \cite{br-ge:ait}*{Theorem 7.2}
and in \cite{bro:fpo}*{Theorem 4.21}.

\end{document}